\documentclass[12pt]{amsart}
\usepackage{amsmath,amssymb,amsfonts}
\usepackage{mathtools}
\usepackage[shortlabels]{enumitem}
\setlist{leftmargin=*}

\usepackage{tikz-cd}

\makeatletter
\def\@tocline#1#2#3#4#5#6#7{\relax
  \ifnum #1>\c@tocdepth 
  \else
    \par \addpenalty\@secpenalty\addvspace{#2}%
    \begingroup \hyphenpenalty\@M
    \@ifempty{#4}{%
      \@tempdima\csname r@tocindent\number#1\endcsname\relax
    }{%
      \@tempdima#4\relax
    }%
    \parindent\z@ \leftskip#3\relax \advance\leftskip\@tempdima\relax
    \rightskip\@pnumwidth plus4em \parfillskip-\@pnumwidth
    #5\leavevmode\hskip-\@tempdima
      \ifcase #1
       \or\or \hskip 1em \or \hskip 2em \else \hskip 3em \fi%
      #6\nobreak\relax
    \dotfill\hbox to\@pnumwidth{\@tocpagenum{#7}}\par
    \nobreak
    \endgroup
  \fi}
\makeatother

\usepackage[sorted]{amsrefs}

\newtheorem{thm}{Theorem}[section]
\newtheorem{cor}[thm]{Corollary}
\newtheorem{lem}[thm]{Lemma}
\newtheorem{prop}[thm]{Proposition}
\newtheorem{claim}[thm]{Claim}
\newtheorem{fact}[thm]{Fact}
\theoremstyle{definition}
\newtheorem{defn}[thm]{Definition}
\newtheorem{rem}[thm]{Remark}
\newtheorem{sample}[thm]{Example}
\numberwithin{equation}{section}

\newcommand{\RR}{\mathbb{R}}
\newcommand{\CC}{\mathbb{C}}

\newcommand{\CP}{\mathcal{P}}

\gdef\g{\mathfrak{g}}
\gdef\h{\mathfrak{h}}
\newcommand{\QQ}{\mathbb{Q}}
\def\CF{\mathcal F}
\newcommand{\ZZ}{\mathbb{Z}}
\newcommand{\CR}{\mathcal R} 

\newcommand{\co}{\circ}
\newcommand{\la}{\langle}
\newcommand{\ra}{\rangle}
\def\tp{{\mathrm{tp}}}
\gdef\CL{\mathcal{L}}
\gdef\st{\operatorname{st}}
\gdef\cl{\operatorname{cl}}
\def\CA{\mathcal A}
\def\CO{\mathcal O}

\def\Ran{\mathbb{R}_{\mathrm{an}}}

\gdef\CX{\mathcal{X}}

\gdef\Stab{{\mathrm{Stab}}}
\gdef\UT{\operatorname{UT}}
\gdef\ut{\operatorname{\mathfrak{ut}}}
\gdef\Rom{\RR_\mathrm{om}}

\usepackage[colorlinks=true, linkcolor=blue]{hyperref}

\title{O-minimal flows on nilmanifolds}
\author[Y.Peterzil]{Ya'acov Peterzil}
\address{University of Haifa}
\email{kobi@math.haifa.ac.il}
\author[S.Starchenko]{Sergei Starchenko}
\address{University of Notre Dame}
\email{sstarche@nd.edu}

\begin{document}

\thanks{Both authors thank Institute Henri Poincare in
Paris, for its hospitality during the ``Model theory, Combinatorics and valued
fields'' term in the Spring trimester of 2018. They also thank the mathematical institute CIRM in Trento
  for its hospitality in Summer of 2019.}
\thanks{The first author was partially supported by ISF grant 290/19.}
\thanks{The second author was
partially supported by the NSF research grant DMS-1500671.}

\begin{abstract} Let $G$ be a connected, simply connected nilpotent Lie group, identified with
a real algebraic subgroup of $\UT(n,\RR)$, and let $\Gamma$ be a lattice in $G$,
with $\pi:G\to G/\Gamma$ the quotient map. For a semi-algebraic $X\subseteq G$, and more
generally a definable set in an o-minimal structure on the real field, we consider
the topological closure of $\pi(X)$ in the compact nilmanifold $G/\Gamma$.

 Our theorem describes $\cl(\pi(X))$ in terms of finitely many
families of cosets of real algebraic subgroups of $G$. The underlying families are
extracted from $X$, independently of $\Gamma$. We also prove an
equidistribution result in the case of curves.
\end{abstract}

\maketitle
\tableofcontents

\section{Introduction}
 Let $\UT(n,\RR)$ denote the group of real $n\times n$ upper triangular
matrices with $1$ on the diagonal. Below we say that a group $G$ is \emph{a real
unipotent group} if it is a real algebraic subgroup of $\UT(n,\RR)$, namely a
subgroup of matrices which is a solution set to a system of real polynomials in the
matrix coordinates. Such subgroups are exactly the connected Lie subgroups of
$\UT(n,\RR)$, and every connected, simply connected nilpotent Lie group is Lie
isomorphic to a real unipotent group. For $\Gamma$ a discrete co-compact subgroup of
real unipotent $G$, the compact manifold $G/\Gamma$ is called \emph{a compact
nilmanifold}. We let $\pi:G\to G/\Gamma$ be the map $\pi(g)=g\Gamma$.

 Let $\Rom$ be an o-minimal
expansion of the real field and $G$ a real unipotent group. We consider the
following problem:

\medskip
\emph{Given $X\subseteq G$ an $\Rom$-definable set (e.g. $X\subseteq G$ a semi-algebraic set),
what is the topological closure of $\pi(X)$ in the nilmanifold $G/\Gamma$?}
\medskip

A special case of this problem is when the set $X\subseteq G$ is the image of $\RR^d$
under a polynomial map (with $G$ viewed in an obvious way as a subset of
$\RR^{n^2}$).
 In \cite{Shah} Shah considers a similar question when $G$
 is an arbitrary real
algebraic linear group, and in \cite{Leibman1} Leibman considers a discrete variant
of the problem,  when $X$ is the image of $\ZZ^d$ under certain polynomial maps inside
nilpotent  Lie groups. Both prove results about
 equidistribution
from which theorems about the closure of $\pi(X)$ can be deduced.  Our setting is
more general, but the results we obtain answer mostly the
 closure problem. In Theorem~\ref{Leibman} below and in Section~\ref{Leibman-sec} we
 show how to deduce closure results similar to theirs from our work.

In order to state our main theorem we set some notation: We fix $G$ a real unipotent
group and  $\Rom$ an o-minimal expansion of the real field. Given a lattice $\Gamma$
in $G$, namely a discrete co-compact subgroup of $G$, we denote by
$M^G_\Gamma=G/\Gamma$ the associated compact nilmanifold and by $\pi^G_\Gamma:G\to
M^G_\Gamma$ the quotient map $\pi^G_\Gamma(g)=g\Gamma$. We omit $G$ from the
notation when the context is clear. Given an $\Rom$-definable set $X\subseteq G$, we want
to describe the topological closure of $\pi_\Gamma(X)$  in $M_\Gamma$.

As we shall see, the frontier of $\pi_\Gamma(X)$ is given via families of orbits of
real algebraic subgroups of $G$ in $M_\Gamma$. For that we make use of the following
theorem, which can be viewed as a special case of our problem when $X$ is a real
algebraic subgroup of $G$. For the discrete one-variable case, see Lesigne
\cite{Lesigne}, and for the more general result about closures of orbits of
unipotent groups, see Ratner \cite{Ratner}.
\begin{thm} [\cite{Lesigne},\cite{Ratner}]\label{ratner}
Let $G$ be a real unipotent group. Assume that  $\Gamma$ is a
lattice in $G$. If $H \subseteq G$ is a real algebraic subgroup then there exists a
unique real algebraic group $H_0\supseteq H$ such that
$$\cl(\pi_\Gamma(H))=\pi_\Gamma(H_0).$$

The group $H_0$ is the smallest real algebraic subgroup of $G$ containing $H$ such
that $\Gamma\cap H_0$ is co-compact in $H_0$.
\end{thm}

Let us set aside a specific notation for the above $H_0$:

\begin{defn} Given $H\subseteq G$  real unipotent groups and $\Gamma$ a lattice
in $G$, we let $H^\Gamma$ denote the smallest real algebraic subgroup of $G$
containing $H$ such that $H^\Gamma\cap \Gamma$ is co-compact in $H^\Gamma$.
\end{defn}

We can now state our main theorem:

\begin{thm}\label{thm-main2.5} Let $G$ be a real unipotent group
 and  let $X\subseteq G$ be an  $\Rom$-definable set. Then,
there are finitely many real algebraic subgroups $L_1,\ldots, L_m\subseteq G$ of positive
dimension, and finitely many $\Rom$-definable closed sets $C_1,\ldots, C_m\subseteq G$,
such that for every lattice $\Gamma\subseteq G$, we have:
$$\cl\Bigl(\pi_\Gamma( X)\Bigr)=\pi_{\Gamma}\Bigl(\cl(X) \cup \bigcup_{i=1}^m C_i L_i^\Gamma\Bigr).$$
In addition, we may choose the sets $C_i$  so that:
\begin{enumerate}
\item For every $i=1,\ldots, m$, $\dim(C_i)<\dim X$.
\item Let $L_i$ be  maximal  with respect to inclusion  among $L_1,\ldots, L_m$.
Then $C_i$ is a bounded subset of $G$, and in particular,
$\pi_\Gamma(C_iL_i^\Gamma)$ is closed in $M_\Gamma$.
\end{enumerate}
\end{thm}

As an immediate corollary we obtain:
\begin{cor} For $G$ real unipotent and $X\subseteq G$ an $\Rom$-definable set,
 if $\Gamma\subseteq G$ is a lattice then there exists an $\Rom$-definable
set $Y\subseteq G$, such that
$$\cl(\pi_{\Gamma}(X))=\pi_\Gamma(Y).$$
\end{cor}

As part of our analysis we conclude in Section~\ref{Leibman-sec}
  the following variant of theorems of Shah and Leibman:
  \begin{thm}\label{Leibman}
Let  $G$ be a unipotent group, viewed as a subset
of $\RR^{n^2}$, and
$F\colon \RR^d \to \RR^{n^2}$ a polynomial map that takes values in $G$.
Let $X\subseteq G$ be the image of $\RR^d$ under $F$.
If $cH\subseteq G$ is the smallest coset of a real algebraic subgroup of
$G$ with
$X\subseteq cH$ then for every lattice $\Gamma\subseteq G$
$$\cl(\pi_\Gamma(X))=\pi_\Gamma(cH^\Gamma).$$
\end{thm}

Finally, in Section~\ref{sec:unif-distyr}, we prove an equidistribution result for definable curves on $G$.
Recall that $\Rom$  is called {\em polynomially bounded} if every definable $1$-variable function is bounded at $\infty$ by some polynomial. We let $G$ be a unipotent group, $\Gamma$ a lattice on $G$ and  $\pi:G\to G/\Gamma$ the natural projection.
Let $\mu_{G/\Gamma}$ be the (unique) $G$-invariant  probability measure on $G/\Gamma$. We prove
\begin{thm}\label{equidist} Assume that $\Rom$ is polynomially bounded, and $\gamma:(0,\infty)\to G$ a definable curve such that $Im(\gamma)\Gamma$ is dense in $G$.
Then for every continuous function $f:G/\Gamma\to \mathbb R$,
\[\lim_{T\to \infty} \frac{1}{T}\int_0^T f(\pi(\gamma(t)))dt =\int_{G/\Gamma} f d\mu_{G/\Gamma}. \]

\end{thm}
We also show that the assumption on $\Rom$ is necessary.

\vspace{.2cm}

Here are  some comments on Theorem~\ref{thm-main2.5}:
\begin{rem}
 \begin{enumerate}
 \item If we let
$X$ be a definable curve, i.e. $\dim(X)=1$, then by Theorem~\ref{thm-main2.5}(1)
there are
 finitely many  real algebraic
subgroups $L_1,\ldots, L_m$, determined by the curve $X$, and finitely many points
$c_1,\ldots, c_m\in G$ such that for every lattice $\Gamma\subseteq G$,
$$\cl(\pi_\Gamma(X))=\pi_\Gamma(X) \cup \, \bigcup_{i=1}^m
\pi_\Gamma(c_iL_i^\Gamma).$$
Thus the closure of  $\pi_\Gamma(X)$ is obtained by
attaching to it  finitely  many sub-nilmanifolds of $G/\Gamma$ (we recall below the
definition of a sub-nilmanifold).
\item In \cite{o-minflows} we examined the same problem in the special
  case when $G$ was abelian, so could be identified with $\la \RR^n,+\ra$ and the final theorem  was very similar to the
current one. We also proved there a finer theorem when  $G=\la \CC^n,+\ra$ and
$X\subseteq \mathbb C^n$ a complex algebraic variety. That work was inspired by questions
of Ullmo and Yafaev in \cite{UY} and \cite{flow}.
\item In the same paper \cite{o-minflows}  we showed that one cannot in general replace the sets $C_i$
in Theorem~\ref{thm-main2.5} by finite sets. For a simple example (pointed out to us by
Hrushovski) one can just start with the curve $C=\{(t,1/t):t>1\}$ in $\RR^2$ and
then consider $\pi_{\ZZ^4}(C\times C)$ inside $\RR^4/\ZZ^4$. If we let $H=\RR\times
\{0\}$, then the frontier of $\pi_{\ZZ^4}(C\times C)$ equals
$$ \pi_{\ZZ^4}((C\times H)\cup (H\times C)\cup (H\times H)).$$
\end{enumerate}
\end{rem}

We end this introduction by noting that definable sets in o-minimal structures allow
for a richer collection than semialgebraic sets, and thus for example we could take
$X\subseteq \UT(3,\RR)$ to be the following
 $\RR_\mathrm{an,exp}$-definable set
 $$\left\{\left(%
\begin{array}{ccc}
  1 & e^y & \arctan(y) \\
  0 & 1 & 1/\sqrt{x^2+y^4} \\
  0 & 0 & 1 \\
\end{array}%
\right):  x,y>0 \, \right \}.$$

\subsection{On definable subsets of arbitrary nilpotent Lie groups}\label{sec-alternative}
Instead of working with real unipotent groups we could have worked in a more general setting:

Let $G$ be a  connected, simply connected nilpotent Lie group. It is
known (e.g.\  see
\cite{nilpotent-book}) that  $G$ is Lie isomorphic to a real algebraic subgroup
$G_0$ of $\UT(n,\RR)$. Given an o-minimal structure $\Rom$, we may declare a subset
of  $G$ to be $\Rom$-definable (or real algebraic) if its image under the above
isomorphism is an $\Rom$-definable (or real algebraic) subset of $G_0$. As noted in
Lemma~\ref{homom2} below, every Lie isomorphism between real unipotent groups is
given by a polynomial map and thus this notion of definability (or algebraicity)
does not depend of the choice of $G_0$ or the isomorphism between $G$ and $G_0$. It
follows from  Fact~\ref{equiv} below that every closed connected subgroup of $G$ is
algebraic in this sense, and thus Theorem~\ref{thm-main2.5} holds for an arbitrary
connected, simply connected nilpotent Lie groups, under the above interpretation of
the relevant notions.

\subsection{On the proof}
Our proof combines  model theory with the theory
of nilpotent Lie groups. It breaks down into three main parts.

 Given an  $\Rom$-definable $X\subseteq G$ we examine the contribution of complete types
on $X$ (see Preliminaries for more details on the basic notions) to the closure of
$\pi_{\Gamma}(X)$. To each complete type $p$ on $X$ we assign ``the nearest coset to
$p$'',  a coset of a real algebraic subgroup of $G$, which we denote by $c_pH_p$
(see Section~\ref{sec-nearest}).  We then prove, see Corollary~\ref{maincor-cosets}, that for every lattice
$\Gamma$, the closure of $\pi_\Gamma(X)$ is the union of all
$\pi_\Gamma(c_pH_p^\Gamma)$, as $p$ varies over all complete types
 on $X$.  Notice that the coset $c_pH_p$   is independent of the lattice
$\Gamma$.

Next, in Lemma~\ref{definability}, we use model theory to show that the family of
nearest cosets
$$\{c_pH_p:p \mbox{ a complete type on } X\}$$ is itself a definable family in
$\Rom$.

Finally, we use Baire Category Theorem to obtain finitely many families of fixed
subgroups of $G$.

\section{Preliminaries}

\subsection{Lattices and nilmanifolds}

We list some basic notions and properties of lattices in simply connected nilpotent
Lie groups.  For a reference we use \cite{nilpotent-book} and \cite{Leibman}.

We identify the Lie algebra of $\UT(n,\RR)$ with $\ut(n,\RR)$, the space of real
$n{\times}n$ upper triangular matrices with $0$ on the main diagonal.

The following fact will be used often.

\begin{fact}\label{fact:exp}  The matrix exponential map restricted to
  $\ut(n,\RR)$  is polynomial  and maps $\ut(n,\RR)$ diffeomorphically
  onto   $\UT(n,\RR)$.  Its inverse $\log\colon \UT(n,\RR)\to
  \ut(n,\RR)$ is a polynomial map as well.
\end{fact}

\begin{rem}\label{rem:exp}
  If $G$ is a closed subgroup of $\UT(n,\RR)$ then we identify its
Lie algebra $\g$ with a subalgebra of $\ut(n,\RR)$.  It follows from
Fact~\ref{fact:exp} that if $G$ is a connected closed subgroup of $\UT(n,\RR)$ then
the exponential map $\exp_G\colon \g\to G$ is a polynomial map (in matrix
coordinates) that is also a diffeomorphism, and its  inverse $\log_G\colon G\to \g$
is polynomial as well.
\end{rem}

We  note:
\begin{lem}\label{equiv}
  Assume that $G\subseteq \UT(n,\RR)$ is a subgroup. Then the following are
equivalent:
\begin{enumerate}
\item $G$ is a closed, connected subgroup of $\UT(n,\RR)$.
\item $G$ is a real algebraic subgroup of $\UT(n,\RR)$.
\item $G$ is definable in $\Rom$.
\end{enumerate}
\end{lem}

\begin{proof} The equivalence of (1) and (2) follows from the fact the exponential map and
its inverse are polynomial maps.

Clearly, every real algebraic subgroup of
$\UT(n,\RR)$ is $\Rom$-definable, so $(2)\Rightarrow (3)$.

To see that $(3)\Rightarrow (1)$, note that every
definable  in an o-minimal structure subgroup is closed and has finitely
many connected components \cite{pillay}.  Let $G^0$ be the connected
component of $G$ containing the identity $e$.  Since $G$ is torsion
free, by \cite{strz}, its o-minimal Euler characters $\chi(G)$ is $+1$
or $-1$.  Since o-minimal Euler characteristic is additive and
invariant under definable bijections, we have $[G:G^0]\chi(G_0)=\pm
1$.  Hence $[G:G^0]=1$, $G=G^0$ and $G$ is connected.
\end{proof}

For the rest of this section we assume that $G$ is a real unipotent group, namely a
real algebraic subgroup of $\UT(n,\RR)$,  with $\mathfrak g$ its Lie algebra. Since
$\exp_G:\mathfrak g\to G$ is a diffeomorphism, the group $G$ is a simply connected,
and we have (\cite[Corollary~5.4.6]{nilpotent-book}):
\begin{fact} A discrete subgroup $\Gamma\subseteq G$ is co-compact
  (i.e. $G/\Gamma$ is compact) if and only if there is a $G$-invariant
probability measure  on $G/\Gamma$.
\end{fact}

\begin{defn}  A subgroup $\Gamma$ of $G$ is called \emph{a lattice in G} if $\Gamma$ is discrete and co-compact.
 If $\Gamma$ is a lattice in $G$ then the quotient $G/\Gamma$ is called \emph{a compact
nilmanifold}.

Given a lattice $\Gamma\subseteq G$,  a real algebraic subgroup $H$ of $G$ is called a
\emph{$\Gamma$-rational} if $\Gamma\cap H$ is a lattice in $H$.
\end{defn}

\begin{rem}\label{rem:uniform}
In \cite{nilpotent-book} a closed subgroup $H$ of $G$ is
defined to be $\Gamma$-rational if the Lie algebra $\h$ of $H$ has a
 basis in the $\QQ$-linear span of $\log_G(\Gamma)$.  By
\cite[Theorem~5.1.11] {nilpotent-book} these two definitions are equivalent.
\end{rem}

 The following is easy to verify:
\begin{fact}\label{Lambda1}
If $\Gamma$ is a lattice in $G$ then there is no real
algebraic subgroup of $G$ containing $\Gamma$ other than $G$.
\end{fact}

\begin{fact}\label{rational-facts}
Let $H\subseteq G$ be a real algebraic normal subgroup with $\pi:G\to G/H$ the quotient map. Let
$\Gamma\subseteq G$ a discrete subgroup. Then: \begin{enumerate} \item If $\Gamma$ is a
lattice in $G$ and $\Gamma\cap H$ is a lattice in $H$ then $H\Gamma$ is closed in
$G$ and $\pi(\Gamma)$ is a lattice in $G/H$.

\item If $\Gamma\cap H$ is a lattice in $H$ and $\pi(\Gamma)$ is a lattice in $G/H$
then $\Gamma$ is a lattice in $G$.

\item If $\Gamma$ is a lattice in $G$ then $H$ is $\Gamma$-rational if and only if
$\pi_\Gamma(H)$ is closed.

 \item If $\Gamma$ is a lattice in $G$ then all subgroups in the ascending central
series are $\Gamma$-rational, in particular $Z(G)$ is $\Gamma$-rational. Also,
$[G,G]$ and all subgroups in the descending central series are  $\Gamma$-rational
subgroups (in particular closed).

\item If $\Gamma$ is a lattice in $G$ and  $H_1, H_2\subseteq G$ are real algebraic
$\Gamma$-rational subgroups then so is $H_1\cap H_2$.
\end{enumerate}
\end{fact}
\begin{proof}
$(1)$ and $(2)$ follow from    \cite[Lemma~5.1.4] {nilpotent-book}.

$(3)$.  If $H$ is $\Gamma$-rational then $H\Gamma$ is closed in $G$ by
$(1)$. Assume $H\Gamma$ is closed in $G$. Then $\pi_\Gamma(H)$ is
closed in $G/\Gamma$, hence compact.  We can find then a compact
subset $K\subseteq H$ such that    $\pi_\Gamma(K)=\pi_\Gamma(H)$, i.e.
$K\Gamma =H\Gamma$. It is not hard to see that $K\Gamma$ is closed,
since it is a product of compact
 and closed sets.

$(4)$ follows from \cite[Proposition~5.2.1] {nilpotent-book}.

$(5)$ follows from Remark~\ref{rem:uniform}.  Indeed, since  $H_1$ and
$H_2$ are $\Gamma$-rational their Lie algebras $\h_1$ and $\h_2$ both have basis in the $\QQ$-vector
space $\QQ$-span of $(\log_G(\Gamma)$.  The Lie algebra of $H_1\cap
H_2$ is $\h_1\cap \h_2$ and it has basis in the same $\QQ$-vector
space.
\end{proof}

We shall also need the following:

\begin{lem}\label{rational}
Let  $\Gamma\subseteq G$ be a lattice in $G$.
Let $H$ be a real algebraic normal subgroup of $G$. Then $H^\Gamma$ is also normal
in $G$.
\end{lem}

\begin{proof}
Since $H$ is invariant by conjugation, and every $\Gamma$-conjugate of
$H^\Gamma$ is also $\Gamma$-rational, it follows that $H^\Gamma$ is normalized by
$\Gamma$. Thus the normalizer of $H^\Gamma$ is a real algebraic subgroup containing
$\Gamma$, so by Fact~\ref{Lambda1} equals to $G$.
\end{proof}

\begin{defn} Let $M=G/\Gamma$ be a compact nilmanifold. A set $S\subseteq M$ is called a \emph{sub-nilmanifold of $N$} if
there exists $a\in G$ and a $\Gamma$-rational group $H\subseteq G$ such that
\[ S=\pi_\Gamma(aH).\]
\end{defn}

The group $G$ acts on $M$ on the left and  the sub-nilmanifold $S$ can
also be written as $S=a\cdot \pi_\Gamma(H)$.

Note that a sub-nilmanifold of $M$ is closed in $M$ and can be written as an orbit
of the element $\pi_\Gamma(a)$, under the group $aHa^{-1}$.

\medskip
We use the following lemma to identify quotients of unipotent group
with semialgebraic sets:

\begin{lem}\label{complement-to-group} Let $G$ be a real unipotent group and let $H\subseteq G$ be a real algebraic
subgroup. Then there exists a closed semialgebraic set $A\subseteq G$ such that the map
$f:A\times H\to G$ given by $(a,h) \mapsto \cdot h$ is a diffeomorphism.
\end{lem}

\begin{proof} Let $\mathfrak h\subseteq \mathfrak g\subseteq \ut(n,\RR)$ be the Lie algebras of $H$ and
$G$, respectively, and let $n=\dim G$ and $k=\dim H$. By
\cite[Theorem~1.1.13]{nilpotent-book}, there is a weak Malcev basis $\{\xi_1,\ldots,
\xi_n\}$   for $\mathfrak g$ through $\mathfrak h$. Namely, $\{\xi_1,\ldots,
\xi_k\}$ is a basis for $\mathfrak h$, and for every $m\leq n$, the $\RR$-linear
span of $\xi_1,\ldots, \xi_m$ is a Lie subalgebra of $\mathfrak g$.

By \cite[Proposition~1.2.8]{nilpotent-book}, the map $\psi:\RR^n\to G$ defined by
$$\psi(s_1,\dotsc,s_n)=\exp_G(s_1\xi_1)\cdot \,\dotsc\, \cdot
\exp_G(s_n\xi_n)$$
is a polynomial diffeomorphism. It sends $\RR^k\times \{0_{n-k}\}$ onto the group $H$ and the
subspace $\{0_k\}\times \RR^{n-k}$ onto a closed semialgebraic subset of $G$, which
we call $A'$. We have $G=H\cdot A'$, and  if we now let $A=\{a^{-1}:a\in A'\}$ and
replace $\psi(\bar s)$  by $\psi(\bar s)^{-1}$, then we see that $G=A\cdot H$ and
the result follows.\end{proof}

\medskip
Recall that in any nilpotent group $G$, if $H\subseteq G$ is a proper subgroup then $H$
is contained in a proper normal subgroup of $G$. Let us see that this remains true
when restricting to real unipotent groups:

\begin{claim}\label{proper-normal}
  If $G$ is a real unipotent group and $H\subseteq G$ is a  proper
real algebraic subgroup then $H$ is contained in a proper normal real algebraic
subgroup of $G$.
\end{claim}

\begin{proof} By \cite[Theorem 1.1.13]{nilpotent-book}, there is a chain of real algebraic
subgroups, $$\{e\}=H_0\subseteq \, \cdots \, \subseteq H=H_m\subseteq H_{m+1}\,
\subseteq \cdots \,\subseteq  H_n=G,$$ with $n=\dim G$, and $\dim H_{i+1}=\dim H_i
+1 $. It follows from \cite[Corollary 1.15]{nilpotent-book}, that $H_{n-1}$ is
normal in $G$, so we are done.
\end{proof}

\medskip
 Finally, we want to show that the collection of all cosets of real algebraic
subgroup of $G$ is itself a semi-algebraic family. By Fact~\ref{fact:exp},
$\exp:\ut(n,\RR)\to \UT(n,\RR)$ is a polynomial diffeomorphism. It induces a
bijection between the Lie subalgebras of $\ut(n,\RR)$ and the connected closed
subgroups of $\UT(n,\RR)$. Because the family of all Lie subalgebras of $\ut(n,\RR)$
is semi-algebraic we obtain:

\begin{fact}\label{all-cosets}  The family $\CF_n$ of all cosets of real algebraic subgroups of $\UT(n,\RR)$ is
semi-algebraic. Namely, there exists a semi-algebraic set $S\subseteq M_n(\RR)\times
\RR^k$, for some $k$, such that
$$\CF_n=\{P\in M_n(\RR):\exists \bar b\in \RR^k\, (P,\bar b)\in S\}.$$
\end{fact}

In fact, by Definable Choice, we may choose the above family so that every coset is
represented exactly once.

\subsection{Maps between real unipotent groups}\label{sec:maps-between-real}

\begin{defn} Let $G$ be a real unipotent group. A map $f:\RR^d\to G$ is called \emph{polynomial} if, when we view $G$ as
a subset of $\RR^{n^2}$, the coordinate functions of $f$ are real polynomials in
$x_1,\ldots, x_d$.  A map $f:G\to \RR^d$ is \emph{polynomial} if $f$ is the
restriction to $G$ of a polynomial map from $\RR^{n^2}$ into $\RR^d$.
\end{defn}

We note:
\begin{lem}\label{homom2} \begin{enumerate}
\item If $G_1$ and $G_2$ are real unipotent groups and $f:G_1\to G_2$ is a Lie
homomorphism then $f$ is a polynomial map. \item Let $G$ be a real unipotent group
and $v$ an arbitrary element in its Lie algebra $\mathfrak g\subseteq \ut(n,\RR)$. If
$p:\RR^d\to \RR$ is a polynomial function then $f(\bar x)=\exp_G(p(\bar x)v)$ is a
polynomial map from $\RR^d$ into $G$.
\end{enumerate}
\end{lem}
\begin{proof}
(1) By standard Lie  theory we have $f=\exp_{G_2}\co\, df \co \log_{G_1},$ where
$df:\mathfrak g_1\to \mathfrak g_2$ is a linear map. Since $\log_{G_1}$ and
$\exp_{G_2}$ are polynomials, $f$ is polynomial as well.

(2) By Fact~\ref{fact:exp}, the map $\exp:\ut(n,\RR)\to \UT(n,\RR)$ is polynomial,
and $\exp_G$ is its restriction to $\mathfrak g$ is clearly polynomial as well. The
map $f:\RR^d\to G$ is thus a composition of polynomial maps. \end{proof}

\subsection{Model theoretic preliminaries}
We use the same set-up as in \cite[Section 2]{o-minflows}. We refer to \cite{omin}
and \cite{CM} for introductory material on o-minimal structures, as well as
examples. We let
$$\CL_{\mathrm{sa}}=\la +,-,\cdot,<,0,1\ra$$ be the language of
ordered rings (as the subscript suggests,  the definable sets in the ordered field
$\RR$ are the semialgebraic sets). We let $\CL_{\mathrm{om}}\supseteq
\CL_{\mathrm{sa}}$ be the language of our o-minimal structure $\Rom$.  We let
$\CL_{\mathrm{full}}$ be the language
 in
which every subset of $\RR^n$ has a predicate symbol, and let $\RR_{\mathrm{full}}$
be the corresponding structure on $\RR$. Clearly, every $\Rom$-definable set is also
$\RR_{\mathrm{full}}$-definable.

All definable sets are definable \emph{with parameters}. The dimension of a
definable set in an o-minimal structure is defined using the cell decomposition
theorem. In our setting it is enough to know that an  $\Rom$-definable $X\subseteq \RR^n$ has
dimension $k$ if and only if it can be decomposed into finitely many
$C^1$-submanifolds of $\RR^n$, whose maximal dimension is $k$.

\subsubsection{Elementary extensions and some valuation theory}\label{standard-part}
We let $\mathfrak R_{\mathrm{full}}=\la \mathfrak R,\ldots\ra$
 be an elementary extension of $\RR_{\mathrm{full}}$ which is
 $|\RR|^+$-saturated,
{or alternatively an ultra-power of
 $\RR_{\mathrm{full}}$ with respect to an appropriate ultrafilter.

 We let $\mathfrak R_{\mathrm{om}}$ and $\mathfrak R$ be
reducts of $\mathfrak R_{\mathrm{full}}$ to the languages $\CL_{\mathrm{om}}$ and
$\CL_{\mathrm{sa}}$, respectively. Given any set $X\subseteq \RR^n$, we denote by
$X^\sharp=X(\mathfrak R)$ its realization in $\mathfrak R_{\mathrm{full}}$. We use roman
letters $X,Y,Z$ etc.\  to denote subsets of $\RR^n$ and script letters $\mathcal X,
\mathcal Y,\mathcal Z$ to denote subsets of $\mathfrak R^n$ that are
not necessarily  of the form
$X^\sharp$ for some $X\subseteq \RR^n$.

The underlying field  $\la \mathfrak R;+,\cdot \ra$ of $\mathfrak
  R_{\mathrm{full}}$ is real closed and we let
 $$\CO(\mathfrak R)=\{\alpha \in \mathfrak R: \exists n\in \mathbb N\,\,
 |\alpha|<n\}.$$
It is  a valuation ring  of $\mathfrak R$ and  its the maximal ideal
  $\mu(\mathfrak{R})$ is the set of infinitesimal elements, namely
$$\mu(\mathfrak R)=\{\alpha\in \mathfrak R:\forall n\in \mathbb N\, |\alpha|<1/n\}.$$

Mostly, for a linear real algebraic  $G$, we shall use a group variant $\mathcal O(G)$
and $\mu(G)$ of the above, defined as follows.
 Because $G$ is a closed subset of $\operatorname{GL}(n,\RR)$, it can be viewed as a closed subset of
 $\RR^{n^2}$,
 and then $G^\sharp$ is a subset  of $\mathfrak R^{n^2}$. In the definitions below we let $I$ denote the identity
 matrix and use $+$ for the usual addition in $\RR^{n^2}$.

 We let
$$\CO(G)=\CO(\mathfrak R)^{n^2}\cap G^\sharp\,\,\mbox{ and } \mu(G)=(I+\mu(\mathfrak R))\cap
G^\sharp.$$
Both $\CO(G)$ and $\mu(G)$ are subgroups of $G^\sharp$, and $\mu(G)$ is normal
in $\CO(G)$. In fact $\CO(G)$ is a semi-direct product of $\mu(G)$ and $G$, so
given $\beta\in \CO(G)$ there exists a unique $b\in G$ such that
$$\beta\in \mu(G)b=b\mu(G).$$ We call  $b$ \emph{the standard part of $\beta$},
denoted as $b=\st(\beta)$. The map $\st:\CO(G)\to G$ is a surjective group
homomorphism whose kernel is $\mu(G)$. It coincides with the the standard part map
on $\CO({\mathfrak R})^{n^2}$, when restricted to $\CO(G)$. We thus have, for
$g=(g_{i,j})_{1\leq  i,j\leq n}\in G^\sharp$,
$$g\in \CO(G)\Leftrightarrow \forall i,j, \, \,g_{i,j}\in \CO(\mathfrak
R)\Leftrightarrow |g|\in \CO(\mathfrak R),$$ where $|g|$ is the Euclidean norm
computed in $\mathfrak R^{n^2}$.

 For $\mathcal X\subseteq G^\sharp$, we let
$$\st(\mathcal X):=\st(\mathcal X\cap \CO(G)).$$
When our setting is clear we shall omit $G$ from the notation and use $\CO$ and
$\mu$ instead.

\medskip
We shall be using extensively the following simple observation:
\begin{fact}\label{standard}If $X\subseteq G$ is an arbitrary set then
$\cl(X)=\st(X^\sharp)$. In particular, if $\Gamma\subseteq G$ is a subgroup then
$$\cl(X\Gamma)=\st(X^\sharp\Gamma^\sharp).$$
\end{fact}

\subsubsection{Types}  If $\CL_{\bullet}$ is any of our languages then
\emph{an $\CL_{\bullet}$-type} $p(x)$ over $\RR$ is a consistent collection of
$\CL_{\bullet}$-formulas with free variables $x$ and parameters in $\RR$, or
equivalently, a collection of sets defined by  $\CL_{\bullet}$-formulas, such that
the intersection of any finitely many of them is non-empty.  When $p(x)$ contains a
formula saying $x\in X$ then we write $p\vdash X$ and say that \emph{$p$ is a type
on $X$}.

An $\CL_{\bullet}$-type $p(x)$ is \emph{complete} if for every
$\CL_{\bullet}$-definable $X\subseteq \RR^n$, where $n=length(x)$, either $X$ or its
complement belongs to $p$.  For $p(x)$ an $\CL_{\bullet}$-type over $\mathbb R$, we
denote by $p(\mathfrak R)$ its realization in $\mathfrak R$, namely the intersection
of all $X^\sharp$, for $X\in p$.

Given $\alpha \in {\mathfrak R}^n$, we let $\tp_{\bullet}(\alpha/\RR)$ be the collection of all
$\CL_{\bullet}$-definable subsets  $X\subseteq \RR^n$ with $\alpha\in X^\sharp$. It is easily seen
to be a complete type.

For $G$ a real unipotent group,  we denote by $S_G(\RR)$ the collection of all
complete $\CL_{\mathrm{om}}$-types $p$ over $\RR$ such that $p\vdash G$.

 Finally, if $p\in S_G(\RR)$, then we let $\mu\cdot p$ be the
(partial) type whose realization is $\mu(G) \, p(\mathfrak R)$. The type $\mu \cdot
p$ is not a complete type, and we call it \emph{a $\mu$-type}. We identify two
$\mu$-types $\mu\cdot p,\mu\cdot q$ if $\mu(G)\, p(\mathfrak R)=\mu(G)\, q(\mathfrak
R)$. The group $G$  acts on the set of all $\mu$-types on the left, since $g\cdot
(\mu\cdot p) =\mu\cdot (g\cdot p)$. See \cite{mustab} for all the above.

The following definition and subsequent theorem, from \cite{mustab}, will play a
significant role in our proof. Given $p\in S_G(\RR)$, we let
$$\Stab^\mu(p)=\{g\in G(\RR): g\cdot (\mu \cdot p)=\mu\cdot p\}.$$ It is easy to see that
$g\in \Stab^\mu(p)$ if and only if $g$ leaves the set
$(\mu\cdot p)(\mathfrak R)$ invariant, when acting on the left.

The next fact follows from\cite{mustab}.
\begin{fact}\label{mustab-thm}
For every   $p\in S_G(\RR)$, the group $\Stab^\mu(p)$ is
 $\CL_{\mathrm{om}}$-definable over $\RR$. Moreover, if $p$ is
unbounded (namely, $p(\mathfrak R)$ is not contained in $\CO(G)$) then $\dim
(\Stab^\mu(p))>0$.
\end{fact}
 Let us see how \cite{mustab} implies the above: Since every type over $\RR$ is definable, $p$ is definable and, by
    \cite[Claim~3.4]{mustab}, there is a definable reduced type $q$
    with $\mu{\cdot} q=\mu {\cdot} p$.  Now use \cite[Theorem~1.1]{mustab}.

The above theorem holds for arbitrary definable  groups in o-minimal structures, and
then  $\Stab^\mu(p)$ is always torsion-free. However,  when $G$ is a
linear real algebraic group  then necessarily $\Stab^\mu(p)$ is real algebraic, even if
the type $p$ is in a richer language.

\section{The nearest coset of a type}\label{sec-nearest}
The goal of this section is to prove that to each complete $\CL_{\mathrm{om}}$-type
$p$ on a real unipotent group $G$  one can associate a coset $gH$ of a real
algebraic subgroup $H\subseteq G$, which is ``nearest'' to $p$ in a precise sense.

Recall that below we are using $H,G$ etc.\   to denote the $\RR$-points of real
groups, and use $H^\sharp, G^\sharp$ etc to denote $\mathfrak R$-points of the same groups.

\begin{defn} For $G$ a linear real  group, $\alpha\in G^\sharp$, $g\in G$,  and $H\subseteq G$ a
real algebraic subgroup, we say that \emph{$gH$ is near $\alpha$} if $\alpha \in
\mu(G)\, gH^\sharp$.
\end{defn}

Note that  there exists $g\in G$ such that
$gH$ is near $\alpha$ if and only if $\alpha\in \CO(G)\, H^\sharp$. Also, if
$\tp_{\mathrm{sa}}(\alpha/\RR)=\tp_{\mathrm{sa}}(\beta/\RR)$ then $gH$ is near
$\alpha$ if and only if $gH$ is near $\beta$. Our ultimate goal is to
show that in unipotent groups
there exists a minimal coset near  $\alpha$.

\begin{lem}\label{flat-lemma}Let $G$ be a linear real algebraic  group and let
 $H, N\subseteq G$ be real algebraic subgroups with
$N$ normal in $G$. Assume that $\alpha \in H^\sharp$ and there is $b\in G$ such that
the coset $bN$ is near $\alpha$. Then $bN\cap H\neq \emptyset$ and the coset $bN\cap
H$  is near $\alpha$ as well.
\end{lem}
\begin{proof}  We have
  $$\alpha =\epsilon bn $$
for some $\epsilon\in \mu(G)$ and  $n\in N^\sharp$.

We first claim that both $b$ and $\epsilon$ belongs to the group $(NH)^\sharp$. Indeed,
$b=\epsilon^{-1} \alpha n^{-1}$, so $b=\st(\alpha n^{-1})$. The element $\alpha
n^{-1}$ belongs to $(NH)^\sharp$, and since $NH$ is a closed subset of $G$, it follows
from Fact~\ref{standard}, that  $b\in NH$.  Hence, $\epsilon=\alpha n^{-1}b^{-1}$ is in $(NH)^\sharp$ as
well.

Thus, we may work entirely in the group $NH$, so we may assume that $G=NH=HN$.

\begin{claim}\label{flat-claim}If $G=NH$ then
$$\mu(G)=\mu(N)\,\mu(H)=\mu(H)\,\mu(N).$$
\end{claim}
\begin{proof}
 By continuity of multiplication, $\mu(N)\,\mu(H)\subseteq \mu(G)$. For the
opposite inclusion, it is enough to show that for every
$\CL_{\mathrm{om}}$-definable $U\subseteq N$, $V\subseteq H$, neighborhoods of $e$,
we have $\mu(G)\subseteq (U\,V)^\sharp$. For that, it suffices to show that the set
$U\,V$ contains an open neighborhood of $e$ in $G$. This follows from the fact that
the map $(x,y)\mapsto xy$ from $N\times H$ into $G$, is a submersion
at $(e,e)$.
\end{proof}

We are now ready to prove the lemma. We start with $\alpha=\epsilon bn$, with
$\epsilon\in \mu(G)$ and $n\in N^\sharp$. Using the above Claim,
$\epsilon=\epsilon_h\epsilon_n$ with $\epsilon_h \in \mu(H)$ and $\epsilon_n\in
\mu(N)$. We also write $b=b_h b_n$, with $b_h\in H$ and $b_n\in N$. So,
$\alpha=\epsilon_h \epsilon_n b_h n'$, with $n'\in N^\sharp$. Since $N$ is normal,
$\epsilon_n b_h= b_h n^*$, for $n^*\in N^\sharp$, so
$$\alpha =\epsilon_hb_hn^* n'.$$

Clearly,  $b_hn^*n'$ is in $b_hN^\sharp$ and since $\alpha$ and $ \epsilon_h$ are in
$H^\sharp$, we also have $b_hn^*n'\in H^\sharp$. So,  $\alpha\in
\mu(G)\,(b_hN^\sharp\cap H^\sharp)$, and in particular $b_hN\cap H$ is nonempty, and
hence a left coset of $N\cap H$. This ends the proof of Lemma~\ref{flat-lemma}.
\end{proof}

\begin{cor}\label{flat-cor1} Let $G$ and $H,N\subseteq G$ be as above.
  Assume that there are $b,c\in G$ such that the cosets  $bN$ and $cH$
  are near $\alpha$. Then $bN\cap cH \neq \emptyset$ and the coset
  $bN\cap cH$ is near $\alpha$.
\end{cor}

\begin{proof} Note first that for any $\epsilon\in \mu(G)$, $\alpha\in \mu(G)\,(bN^\sharp\cap
cH^\sharp)$ if and only if $\epsilon \alpha \in \mu(G)\,(bN^\sharp\cap cH^\sharp)$. Thus,
 we may replace the assumption that $\alpha\in
\mu(G) cH^\sharp$ by $\alpha \in (cH)^{\sharp}$, so $c^{-1}\alpha \in H^\sharp\cap
\mu(G)\,(c^{-1}bN^\sharp)$. We apply Lemma~\ref{flat-lemma} and conclude that
$c^{-1}\alpha\in \mu(G) (c^{-1}bN\cap H)^\sharp$. It follows that $\alpha\in \mu(G)
(bN\cap cH)^\sharp$. In particular, $bN\cap cH\neq \emptyset$, so it is a left coset
of $N\cap H$.
\end{proof}

We also need:

\begin{lem}\label{cosets} Let $G$ be a linear real algebraic  group and $H\subseteq G$  a real
algebraic subgroup. For $g_1,g_2\in G$, assume that $\mu(G)\,g_1H^\sharp\cap
\mu(G)\,g_2H^\sharp\neq \emptyset$. Then $g_1H=g_2H$.
\end{lem}

\begin{proof} We let $$\alpha=\epsilon_1g_1h_1=\epsilon_2g_2h_2,$$ where $h_1,h_2\in H^\sharp$
and $\epsilon_1,\epsilon_2\in \mu(G)$. It follows that $g_2^{-1}g_1=\epsilon
h_2h_1^{-1}$ for some $\epsilon\in \mu(G)$. But then
$g_2^{-1}g_1=\st(h_2h_1^{-1})\in H$, so $g_1H=g_2H$.
\end{proof}

\begin{rem} Although we proved lemmas~\ref{flat-lemma}--\ref{cosets}
  for linear real algebraic groups, the results  hold for an
arbitrary definable group in an o-minimal structures, with exactly the
same proofs and with $\mu(G)$ defined as in \cite{mustab}.
See \cite{Otero} for more on definable groups in o-minimal structures.
\end{rem}

We are ready to prove the main result of this section.

\begin{thm}\label{minimum}
Let $G$ be a real unipotent group and let $\alpha\in G^{\sharp}$.
\begin{enumerate}
\item If $H_1,H_2$ are real algebraic subgroups of $G$ and
  $g_1,g_2\in G$ such that the cosets $g_1H_1$ and $g_2H_2$ are near
  $\alpha$ then  $g_1H_1\cap g_2H_2\neq \emptyset$ and the coset
$g_1H_1\cap g_2H_2$ is near $\alpha$ as well.
\item There exists a smallest left coset of
real algebraic subgroup of $G$, among all such cosets that are near $\alpha$.
\end{enumerate}
\end{thm}

\begin{proof}
(1) We use induction on $\dim G$ and note that the result is obviously true
when $\dim G=1$.

We may clearly assume that $H_1,H_2$ are both proper subgroups of
$G$. So by
Claim~\ref{proper-normal} there exists a proper normal real algebraic $N_1\subseteq G$
containing $H_1$.  Obviously  $\,g_1N_1$ is near $\alpha$.  By
Corollary~\ref{flat-cor1}, $g_1N_1\cap g_2H_2\neq \emptyset$ and for
$d\in
g_1N_1\cap g_2H_2$ the coset  $d(N_1\cap H_2)$ is near
$\alpha$.

Obviously, for $d\in
g_1N_1\cap g_2H_2$ we have $g_1H_1\cap g_2H_2=g_1H_1 \cap d(N_1\cap
H_2)$.  Replacing $H_2$ by $N_1\cap H_2$ and $g_2$ by $d\in N_1\cap
H_2$, if needed, we may assume that $H_2\subseteq N_1$.

By assumption, $\alpha\in \mu(G)\,g_1H_1^\sharp\cap \mu(G) \,g_2H_2^\sharp,$ so $\alpha\in
\mu(G) g_1N_1^\sharp\cap \mu(G) g_2N_1^\sharp$. By Claim~\ref{cosets}, $g_1N_1=g_2N_1$,
hence $g_1^{-1}g_2\in N$.

We now consider $\alpha'=g_1^{-1}\alpha \in N_1^\sharp$, and  note that
$$\alpha'\in \mu(G)\,H_1^\sharp\cap \mu(G)\,g_1^{-1}g_2H_2^\sharp.$$

(2) The existence of a smallest coset immediately follows from (1).
\end{proof}

The above theorem allows us to define:

\begin{defn} Given real unipotent
$G$, and $\alpha \in G^\sharp$, we denote by $A_\alpha$ the smallest coset near
$\alpha$. We call it \emph{the nearest coset to $\alpha$}. We denote by $H_\alpha$
the associated group, so $A_\alpha=gH_\alpha$ for any $g\in A_\alpha$. For $p$ the
complete type $\tp_{\mathrm{om}}(\alpha/\mathbb R)$, we also use $A_p:=A_\alpha$ and
write $A_p=gH_p$.
\end{defn}

Note that if $\alpha \in \CO(G)$
 then the nearest coset to $\alpha$ is just
$\{\st(\alpha)\}$, which can be viewed as a coset of the identity of $G$. On the
other hand, if $\alpha\notin \CO(G)$ then no element in $G$ is near $\alpha$ and
therefore $\dim A_\alpha>0$. We thus have:
\begin{lem}\label{bounded}
For $\alpha\in G^\sharp$, $\,\, \alpha\in \CO(G)$ if and only if $A_\alpha =
\{\st(\alpha)\}$.\end{lem}

We also need:
\begin{lem}\label{homom1} Assume that $G$ and $G_1$ are real unipotent groups and
$f:G\to G_1$ is a surjective Lie homomorphism. Then \begin{enumerate} \item
$f(\mu(G))=\mu(G_1)$ and $f(\CO(G))=\CO(G_1)$. \item If $\alpha\in G^\sharp$ and
$\beta=f(\alpha)$ then $f(A_\alpha)=A_\beta$.\end{enumerate}
\end{lem}
\begin{proof}
By Lemma~\ref{homom2}, $f$ is a polynomial map and hence has a natural
extension to $\mathfrak R_\mathrm{om}$, which is still denoted by $f:G^\sharp\to
G_1^\sharp$.

(1)  The map $f$ is continuous and open (by its surjectivity), and hence we have
$f(\mu(G))=\mu(G_1)$ and $f(\CO(G))=\CO(G_1)$.

(2) It follows from (1) that if  $gH$ is near $\alpha$ then $f(gH)$ is near $\beta$,
and therefore $A_\beta\subseteq f(A_\alpha)$. For the opposite inclusion, assume that
$A_\beta=g_1H_1\subseteq G_1$. We have $\beta\in \mu(G_1)A_\beta$ and therefore
$\alpha\in \mu(G) f^{-1}(A_\beta)$ (here we use that $f(\mu(G))=\mu(G_1)$). By the
minimality of $A_\alpha$, we have $A_\alpha\subseteq f^{-1}(A_\beta)$ and therefore
$f(A_\alpha)\subseteq A_\beta$.\end{proof}

We end this section with an example which shows that
Theorem~\ref{minimum} fails for arbitrary linear real algebraic groups.

\begin{sample}
We work with $G=SL(2,\RR)$. For $\varepsilon$ an infinitesimally small element of
$\mathfrak R$, we let \[ \alpha=\begin{pmatrix} \varepsilon & 0 \\ 0 &
\varepsilon^{-1}
\end{pmatrix} \] be an element of $SL(2,\mathfrak R)$. We show that there is no
minimal coset near $\alpha$.

 We denote  by $D$ the diagonal subgroup of $SL(2,\RR)$. Since
 $\alpha \in D^\sharp$, we have that $D$ is a coset
 near $\alpha$.

 Let
\[ b=\begin{pmatrix} 1 & 1 \\ 0 & 1
\end{pmatrix}, \]
and $H$ be the conjugate of $D$ by $b$, namely $H=b^{-1}Db$.

We consider the coset $bH = Db$, and claim that it is near $\alpha$.   Obviously,
the element $\beta=\alpha b$ is in $D^\sharp b=bH^\sharp$, so it is enough to see that
$\alpha\beta^{-1}$ is in $\mu(G)$.
We have $$\alpha \beta^{-1} =\begin{pmatrix} 1 & \varepsilon^2 \\
0 & 1.
\end{pmatrix},$$ clearly in $\mu(G)$.

Thus, both $D$ and $bH$ are near $\alpha$,  but $D \cap bH=D\cap Db =\emptyset$, so
there is no minimal coset near $\alpha$.

\medskip
 The above example takes place entirely in the solvable group of
upper triangular matrices, thus we see that Theorem~\ref{minimum} fails even for
solvable linear Lie groups.
\end{sample}

\section{The algebraic normal closure of a set}
We still assume here that $G$ is a real unipotent group. All definability is in
$\Rom$.

\begin{defn} Given a definable set $X\subseteq G$, we let $\la X\ra_{alg}$ be the minimal
real algebraic subgroup of $G$ containing $X$.

We call the smallest algebraic normal subgroup of $G$ containing $X$ \emph{the
algebraic normal closure of $X$}.
\end{defn}

\begin{lem}\label{intersection} Let $P\subseteq G$ a  real algebraic subgroup and assume that $U\subseteq G$ is a
 nonempty  open subset of
 $G$. Then the group $\la \bigcup_{g\in U} P^g\ra_{alg}$ is normal in $G$, and in
 particular, equals the algebraic normal closure of  $P$.
 \end{lem}
 \begin{proof}
For subsets $A,S\subseteq G$ we write $A^S$ for $\bigcup_{g\in S} A^g
=\bigcup_{g\in S} g^{-1}Ag$.

Let $u\in U$. Since $\la P^U\ra_{alg}= \la \ (P^u)^{u^{-1}U}\ra_{alg}$, replacing
$P$ by $P^u$ and $U$ by $u^{-1}U$, if needed, we may assume that $U$ is an open
neighborhood of $e$.

Clearly, for $V\subseteq V_1 \subseteq G$ we have $\la  P^V\ra_{alg} \subseteq  \la
P^{V_1}\ra_{alg}$, and  $\la  P^V\ra_{alg}$ is normal in $G$ if and only if $\la
P^V\ra_{alg} = \la  P^G\ra_{alg}$. Thus, to show that $\la P^U\ra_{alg}$ is normal
in $G$, it is sufficient to find a non-empty $B\subseteq U$ such that $\la
P^B\ra_{alg}$ is normal in $G$.

By DCC on real algebraic subgroups, we can find an open neighborhood $U_0$ of $e$
with $U_0 \subseteq U$ such that $\la  P^V\ra_{alg} =  \la P^{U_0}\ra_{alg}$ for any
open neighborhood $V$ of $e$ with $V \subseteq U_0$.  Let $N=\la P^{U_0}\ra_{alg}$.
We claim that $N$ is normal in $G$.

Indeed, choose open $B\ni e$ with. $B^{-1}=B$ and $B B\subseteq U_0$.  Since for any
$b\in B$ we have $e\in Bb\subseteq U_0$, it follows that
\[ N^b= (\la P^{B} \ra_{alg})^b =   \la P^{Bb} \ra_{alg} =N.  \]
Thus the normalizer of $N$ contains an open neighborhood of $e$ and therefore equals
the whole of $G$, hence $N$ is normal in $G$.
\end{proof}

As a corollary we obtain the following proposition. Recall that for a subgroup
$N\subseteq G$ and a lattice $\Gamma\subseteq G$, the group $N^\Gamma$ is the smallest
$\Gamma$-rational subgroup of $G$ containing $N$.

\begin{prop}\label{neat} Let $G$ be a real unipotent group, $P$ a real algebraic subgroup of  $G$,
and $N$ be the algebraic
  normal closure of $P$. Let $\Gamma$ be a lattice in $G$. Then the
  set $X=\{g\in G:(P^g)^{\Gamma}=N^{\Gamma}\}$ is dense in
  $G$.
\end{prop}
\begin{proof} It is sufficient   to prove that the complement of $X$ is nowhere dense in
$G$. Since every conjugate of $P$ is contained in $N$,  this complement can be
written as the  union  over all proper $\Gamma$-rational subgroup $L$ of
$N^{\Gamma}$, of the semialgebraic sets
$$X_L=\{g\in G:(P^g)^{\Gamma}\subseteq L\}=\{g\in G: P^g \subseteq L\}.$$

By Remark~\ref{rem:uniform},  there are at most countably many $\Gamma$-rational
subgroups of $G$, so  by Baire Category Theorem, it is enough to prove that each
of the sets $X_L$ is nowhere dense. Since $X_L$ is semialgebraic  we just need to
see that it does not contain any nonempty open set.

Assume towards contradiction that for some proper $\Gamma$-rational subgroup $L\subseteq
N^\Gamma$, $X_L$ contained an open set $U$. Then $\la \bigcup_{g\in U}P^g\ra_{alg}$
is contained in $L$. But, by Lemma~\ref{intersection}, $\la \bigcup_{g\in
U}P^g\ra_{alg}=N$, so $N\subseteq L$ and hence $N^\Gamma\subseteq L$, contradicting our choice
of $L$.\end{proof}

\section{The main result for complete types}
We assume in this section that $G$ is a real unipotent group.

\begin{lem}\label{claim1} Let $H$ be a real unipotent group,
$f:G\to H$ a surjective homomorphism of Lie groups, and   $\mathcal X$ a subset of
$G^\sharp$.

Then, for every lattice  $\Gamma\subseteq G$, if $f(\Gamma)$ is closed in $H$ then
$$f(\st(\mathcal X \Gamma^\sharp))=\st(f(\CX) f(\Gamma^\sharp)).$$
\end{lem}
\begin{proof} By Lemma~\ref{homom2}, $f$ is polynomial so in particular definable in
$\Rom$. By Lemma~\ref{homom1},  $f$ sends $\CO(G)$ to $\CO(H)$ and $\mu(G)$ to
$\mu(H)$. It follows that for $\alpha\in \CO(G)$ we have
$f(\st(\alpha))=\st(f(\alpha))$.

Let $D_{\CX,\Gamma}=\st(\mathcal X\,\Gamma^\sharp)$. We need to show that
$f(D_{\CX,\Gamma})=\st(f(\CX) f(\Gamma^\sharp))$.

\medskip

$\subseteq$: If $a=\st(\alpha \gamma^*)\in D_{\CX,\Gamma}$, with $\alpha\in \CX$ and
$\gamma^*\in \Gamma^\sharp$ then $f(a)=\st(f(\alpha
\gamma^*))=\st(f(\alpha)f(\gamma^*)) \in \st(f(\CX) \,f(\Gamma^\sharp))$.

$\supseteq$: Assume that $a_1=\st(f(\alpha)\,f(\gamma^*)),$ for some $\alpha\in \CX$
and $\gamma^*\in \Gamma^\sharp$. We want to show that $a_1\in f(D_{\CX,\Gamma})$.

Since $G/\Gamma$ is compact, there exists a compact semi-algebraic set $K\subseteq G$
such that for every $g\in G$, there exists $\gamma\in \Gamma$ with $g\gamma\in K$.
This remains true for $G^\sharp$, $\Gamma^\sharp$ and $K^\sharp$.  Thus, we can find
$\gamma_1^*\in \Gamma^\sharp$ such that $$(\alpha \gamma^*)\gamma_1^*\in K^\sharp
\subseteq \mathcal O(G).$$ We may therefore  take the standard part and get
$a:=\st(\alpha \gamma^*\gamma_1^*)\in D_{\CX,\Gamma}$. It follows that
$$f(a)=f(\st(\alpha\gamma^*\gamma_1^*))=\st(f(\alpha\gamma^*\gamma_1^*))=\st(f(\alpha)
f(\gamma^*\gamma_1^*))\in \st(f(\CX)\,f(\Gamma^\sharp)).$$

Writing $f(a)$ differently we have
$$f(a)=\st(f(\alpha\gamma^*)f(\gamma_1^*))=\st(f(\alpha\gamma^*))\st(f(\gamma_1^*)).$$
Note that we are allowed to write this since indeed $f(\gamma_1^*)\in \CO(H)$,
because both $f(a)$ and $f(\alpha\gamma^*)$ are in $\CO(H)$). So, the term on the
right equals $a_1 \st(f(\gamma_1^*))$.

Finally, since $f(\Gamma)$ is closed in $H$, we have
$$f(\Gamma)=\st(f(\Gamma)^\sharp)=\st(f(\Gamma^\sharp)),$$ hence
$\st(f(\gamma_1^*))=f(\gamma)$, for some $\gamma\in \Gamma$.

Because $D_{\CX,\Gamma}$ is right-invariant under $\Gamma$, its image is
right-invariant under $f(\Gamma)$ and hence $f(a)f(\gamma)^{-1}=a_1$ is in
$f(D_{\CX,\Gamma})$, as we wanted. \end{proof}

\medskip

 Recall that for a complete type $p\in S_G(\RR)$ we let $A_p$ be the nearest
coset to $p$. We can now prove:

\begin{thm}\label{main2}Assume that $p$ is a type in $S_G(\RR)$. Then
  for every lattice $\Gamma\subseteq G$ we have
$$\st(p(\mathfrak R)\Gamma^\sharp)=\cl(A_p\Gamma).$$
\end{thm}
\begin{proof} We write $A_p=gH_p$.  To simplify notation we let
  \[ D_{p,\Gamma}=\st(p(\mathfrak R)\Gamma^\sharp).\]

  We first handle  a special case.

\begin{prop} Assume that $A_p=H_p$ is a subgroup of $G$ and that $H_p^\Gamma=G$.
\label{main-step}  Then  $D_{p,\Gamma}=G$.
\end{prop}

\begin{proof}[Proof of Proposition] We prove the proposition by induction on $\dim G$, starting from $\dim G=0$,
for which the result is trivially true. We assume then that $\dim G>0$.

Since $H_p^\Gamma=G$, the group $H_p$ must have positive dimension, hence $p$ is not
a bounded type, so by Fact~\ref{mustab-thm}, the group $P:=\Stab^\mu(p)$ is a
definable subgroup of positive dimension.

We consider the algebraic normal closure of $P$, call it $N$ and then $N^\Gamma$. By
Lemma~\ref{rational}, $N^\Gamma$ is  normal, hence it is the minimal
normal $\Gamma$-rational subgroup of $G$
containing $P$. Since $G$ is nilpotent, the intersection any nontrivial normal
subgroup with the center  $Z(G)$ is nontrivial (see for example \cite[Proposition
7.13]{stroppel}), so  $N_0=N^\Gamma\cap Z(G)$ is nontrivial. Since $G$ is
torsion-free, $N_0$ is a real algebraic subgroup of positive dimension, so $\dim
G/N_0<\dim G$.

We consider the quotient map
$$f:G\to G/N_0.$$
The group $G/N_0$ is again a connected, simply connected nilpotent Lie group and
hence Lie isomorphic to a real unipotent group. By Lemma~\ref{homom2}, the
composition of this isomorphism with $f$ is a polynomial map. Thus, we identify
$G/N_0$ with a real unipotent group, and still denote the homomorphism from $G$ onto
this unipotent group by $f$.

 We let $q$ be the image of the type $p$ under $f$. By that we mean that
 for some (equivalently any)
$\alpha\in p(\mathfrak R)$ we let $q=\tp_{om}(f(\alpha)/\RR) \vdash G/N_0$. We let
$\Gamma_1=f(\Gamma)$. Since both $Z(G)$ and $N^\Gamma$ are $\Gamma$-rational then so
is $N_0$. It follows that $\Gamma_1$ is a lattice  in $G/N_0$ (for both, see Fact~\ref{rational-facts}).

 Let $A_q=g_q H_q$  be the nearest coset of $q$.  We claim that $A_q^{\Gamma_1}=G/N_0$, namely $H_q^{\Gamma_1}=G/N_0$. Indeed, first
note that by Lemma~\ref{homom1}, we have
 $f(A_p)=A_q$,
so $f(H_p)=A_q$ and hence $A_q=H_q$ is a group. Next, since $N_0$ is
$\Gamma$-rational the pre-image under $f$ of the $\Gamma_1$-rational group
$H_q^{\Gamma_1}$ is a $\Gamma$-rational subgroup of $G$ containing $H_p$, so by our
assumptions on $p$ it equals to $G$. It follows that $H_q^{\Gamma_1}=G/N_0$.

Since $\dim G/N_0<\dim G$, we may apply induction  to $q\vdash G/N_0$ and $\Gamma_1$
and conclude that $\st(q(\mathfrak R)\Gamma_1^\sharp)=G/N_0$. Therefore, by  Lemma~\ref{claim1},
$$f(D_{p,\Gamma})=G/N_0.$$

Next, we    claim that   $D_{p,\Gamma}$ is left-invariant under $P=\Stab^\mu(p)$.
Indeed, if $a\in D_{p,\Gamma}=\st(p(\mathfrak R)\Gamma^\sharp)$ then $a=\epsilon \alpha \gamma^*$ for
$\epsilon\in \mu(G)$, $\alpha\in p(\mathfrak R)$ and $\gamma^*\in\Gamma^\sharp$. By
definition, for every $h\in P$, there exists $\epsilon'\in \mu(G)$ and $\alpha'\in
p(\mathfrak R)$ such that $h\alpha=\epsilon'\alpha'$. But then, for some
$\epsilon''\in \mu(G)$,
$$ha=h\epsilon \alpha \gamma^*=\epsilon'' h \alpha \gamma^*=\epsilon''\epsilon'
\alpha'\gamma^*.$$ Since $ha\in G$, we have $ha =\st(ha)=\st(\alpha'\gamma^*)\in
D_{p,\Gamma}$, so $D_{p,\Gamma}$ is left-invariant under $P$.

By definition, $D_{p,\Gamma}$ is also right-invariant under $\Gamma$.

We now
consider the set
$$Y=\{g\in G: (P^g)^\Gamma=N^\Gamma\}.$$

By Proposition~\ref{neat}, the set $Y$ is dense in $G$.
\\

\noindent{\bf Claim}
\emph{The set $Y$ is contained in $D_{p,\Gamma}$.}

\begin{proof}[Proof of Claim.]
  We will show that $Y\cap D_{p,\Gamma}$ is
left-invariant under $N_0=\ker(f)$ and that $f(Y\cap D_{p,\Gamma})=f(Y)$. The result
follows (since we conclude that $Y= Y\cap D_{p,\Gamma}$).

 First, let us note that $N_0Y=Y$: Because $N_0$ is central, for every $n\in N_0$ and  $g\in G$, $P^g=P^{ng}$,
 so by the definition of $Y$, if $g\in Y$ then so is $ng$.

 In order to show that $Y\cap D_{p,\Gamma}$ is left-invariant under $N_0$ it is enough to show that
 for every $g\in Y\cap D_{p,\Gamma}$,
we have  $N_0g\subseteq D_{p,\Gamma}$. So fix $g\in Y\cap D_{p,\Gamma}$.

Since $D_{p,\Gamma}$ is  left-invariant under $P$ and right-invariant under
$\Gamma$, we have  $Pg\Gamma=gP^g\Gamma \subseteq D_{p,\Gamma}$. Because it is also
closed, we have $\cl(gP^g\Gamma)\subseteq D_{p,\Gamma}$. Since $g\in Y$,
$$\cl(P^g\Gamma)=(P^g)^\Gamma\Gamma=N^\Gamma\Gamma,$$ and hence
$$gN^\Gamma \Gamma=\cl(gP^g\Gamma)\subseteq   D_{p,\Gamma}.$$

 Because $N_0\subseteq N^\Gamma$ and is normal in $G$,  we
have
$$N_0g =gN_0\subseteq gN_0^\Gamma\subseteq D_{p,\Gamma},$$ thus completing the proof
 that $Y\cap D_{p,\Gamma}$ is left-invariant under $N_0$.

Now, since $N_0Y=Y$, we have $f(Y\cap D_{p,\Gamma})=f(Y)\cap f(D_{p,\Gamma})$. We
already saw that $f(D_{p,\Gamma})=G/N_0$, and therefore $f(Y\cap
D_{p,\Gamma})=f(Y)$. Because $Y\cap D_{p,\Gamma}$ is left-invariant under $N_0$ it
follows that $Y\subseteq D_{p,\Gamma}$, completing the proof of the claim.
\end{proof}

Because $Y$ is dense in $G$ and $D_{p,\Gamma}$ is closed we have
 $D_{p,\Gamma}=G$. This ends the proof of Proposition~\ref{main-step}.\end{proof}

 In order to
complete the proof of Theorem~\ref{main2}, consider now an arbitrary type $p\in
S_G(\RR)$, with $A_p=gH_p$. By replacing $p$ with $g^{-1}p$ and $D_{p,\Gamma}$ with
$D_{g^{-1}p,\Gamma}=g^{-1}D_{p,\Gamma}$, we may assume that $A_p=H_p$. For every
$\alpha\in p(\mathfrak R)$ there is $\epsilon\in \mu(G)$ such that $\epsilon
\alpha\in H_p^\sharp$. Since $\st(\epsilon \alpha)=\st(\alpha)$, replacing $\alpha$ with
$\epsilon\alpha$ we may assume that $p\vdash H_p$, and thus $\st(p(\mathfrak
R)\Gamma)\subseteq \cl(H_p\Gamma)=H_p^\Gamma\Gamma$.

Let $G_0=H_p^\Gamma$ and $\Gamma_0=G_0\cap \Gamma$,  a lattice in $G_0$. Notice that
$\cl(H_p\Gamma_0)=H_p^{\Gamma_0}\Gamma_0=H_p^\Gamma=G_0$. Thus, in order to prove
the theorem it is sufficient to show that $\st(p(\mathfrak R)\Gamma_0^\sharp)=G_0$.
This is exactly Proposition~\ref{main-step} (for $G_0$ and $\Gamma_0$ instead of $G$
and $\Gamma$), so we are done.\end{proof}

\medskip
Returning to the setting of Theorem~\ref{thm-main2.5}, we start with a
given definable set $X\subseteq G$, and define the associated family of nearest cosets:
$$\CA(X)=\{A_\alpha:\alpha\in X^\sharp\}.$$

By Lemma~\ref{bounded}, the $0$-dimensional elements of $\CA(X)$ are exactly the
singletons $\{g\}$ for $g\in G$.

 For $\alpha\in X^\sharp$, let $A_\alpha=g_\alpha H_\alpha$, where $g_\alpha$ is any element in $A_\alpha$. For every
  lattice $\Gamma\subseteq G$,
we have
$$\cl(A_\alpha \Gamma)=\cl(g_\alpha H_\alpha \Gamma)= g_\alpha
(H_\alpha)^\Gamma \Gamma.$$

We let $A_\alpha^\Gamma$ denote the coset $g_\alpha H_\alpha^\Gamma$.
 We can now describe the closure of $\pi_{\Gamma}(X)$ as follows:

\begin{cor}\label{maincor-cosets}
For every lattice $\Gamma\subseteq G$,
$$\cl(X\Gamma)=\bigcup_{\alpha\in X^\sharp} g_\alpha (H_\alpha)^\Gamma \Gamma=\bigcup_{\alpha\in X^\sharp} A_\alpha^\Gamma \Gamma,$$
and
$$\cl(\pi_\Gamma(X))=\bigcup_{\alpha\in X^\sharp}\pi_\Gamma(g_\alpha
H_\alpha^\Gamma)=\bigcup_{\alpha\in X^\sharp} \pi_\Gamma(A_\alpha^\Gamma).$$
\end{cor}
\begin{proof} As we saw, $$\cl(X\Gamma)=\st(X^\sharp\Gamma^\sharp)=\bigcup_{p\vdash
X}\st(p(\mathfrak R)\Gamma^\sharp).$$ By Theorem~\ref{main2}, we have
$$\cl(X\Gamma)=\bigcup_{p\vdash X} (A_p^\Gamma)\Gamma.$$
Since $A_\alpha=A_\beta$ whenever $\alpha $ and $\beta$ realize the same complete
type, we can write the same union as $\bigcup_{\alpha\in
X^\sharp}A_\alpha^\Gamma \,\Gamma$. The result follows.\end{proof}

\subsection{An alternative definition of $\CA(X)$.}
\label{sec:an-altern-defin}

In this section we give an alternative definition of $\CA(X)$. This
definition is not used anywhere else, so we will be brief.

As before, $G$ is a real unipotent group.

Viewing  $\mathrm{GL}(n,\RR)$ as a subset of $\RR^{n^2}$, we
denote by $\| \ \|_G$ the restriction of the Euclidean norm on  $\RR^{n^2}$ to
$G$.

For $a,b\in G$ let $d_G(a,b)=\| ab^{-1} - I_n \|_G$.

Let $X\subseteq G$ be a definable set. In this section  by \emph{a definable curve on $X$} we mean a
definable continuous  function $\sigma(t)\colon \RR^{\geq 0}\to X$.

Let $\sigma(t)$ be a definable curve on $G$.  For a coset $aH
\subseteq G$ of a real algebraic group $H$ we say that $aH$ is \emph{near
$\sigma(t)$}  if $\lim_{t\to \infty} d_G(\sigma(t), aH) =0$, where, as
usual,     $d_G(\sigma(t), aH)=\inf\{ d_G(\sigma(t), g) \colon g\in aH
\}$.

Applying Theorem~\ref{minimum} to an infinitely large $t$ we obtain the
following claim.

\begin{claim}
  \label{claim:near-curve}
Let $\sigma(t)$ be a definable curve on $G$.  Let $g_1H_1,
g_2H_2\subseteq G$ be cosets of real algebraic subgroups. If both
$g_1H_1$ and
$g_2H_2$ are near $\sigma(t)$ then $g_1H_1\cap g_2H_2 \neq \emptyset$ and
the coset $g_1H_1\cap g_2H_2$ is near $\sigma(t)$ as well.
\end{claim}

Thus if $\sigma(t)$  is a definable curve on $G$ then there is the
smallest coset near $\sigma(t)$ that we denote by $A_\sigma$ and call
it \emph{the nearest coset to $\sigma(t)$}.

Working in the tame pair $\la \Rom\la \tau\ra, \Rom  \ra$, where
$\Rom\la \tau\ra=\operatorname{dcl} (\Rom\cup \{ \tau \})$ for an
infinitely large $\tau$, we can redefine $\CA(X)$  as follows.

\begin{prop}\label{prop:main-restatement}  Let $X\subseteq G$ be a
  definable subset. Then
  \[
    \CA(X)=\bigcup\{ A_\sigma  \colon \sigma(t)
  \text{ is a definable curve on } X \}.
\]

\end{prop}

\subsection{Digression, the connection to the work of Leibman and Shah}

\label{Leibman-sec}
 Our goal here is to deduce Theorem~\ref{Leibman} from Corollary~\ref{maincor-cosets}. Before doing that, we briefly discuss the
connection between our  notion of ``a polynomial map''
and that of \cite{Leibman1}.

Given $G$ a connected, simply connected nilpotent Lie group,
let $a_1,\ldots, a_n$ be some elements of $G$, and  let $p:\RR^d \to \RR^n$ be a
polynomial map, such that $p(\ZZ^d)\subseteq \ZZ^n$. Then the map $f:\ZZ^d\to G$, defined
by $$f(\bar k)=a_1^{p_1(\bar k)}\cdots a_n^{p_n(\bar k)}$$ is said to be a
polynomial map in \cite {Leibman1}. Note that this definition is invariant under an
isomorphism of $G$ thus we may assume that $G$ is a real unipotent group. By
Lemma~\ref{homom2} (2), there is a map $F\colon \RR^d \to G$, polynomial in matrix
coordinates,  such that $f(\bar k) = F(\bar k)$ for $\bar k  \in \ZZ^d$.

\medskip
We prove:
\begin{thm}\label{Leibman-Shah} Let $G$ be a real unipotent group.
Assume that $f:\RR^d\to G$ is a polynomial map in matrix coordinates and let
$X=f(\RR^d)\subseteq G$. Let $gH$ be the minimal coset among all left cosets of real algebraic subgroups of
$G$ with $X\subseteq gH$. Then for every lattice $\Gamma\subseteq G$,
$$\cl(\pi_{\Gamma}(X))=\pi_\Gamma(gH^\Gamma).$$
\end{thm}
\begin{proof} Note first that for every $\alpha\in X^\sharp$, its nearest coset $A_\alpha$ is
contained in $gH$. Thus, by Corollary~\ref{maincor-cosets}, for every lattice
$\Gamma$,
$$\cl(\pi_\Gamma(X))=\bigcup_{\alpha\in X^\sharp}\pi_\Gamma(A_\alpha^\Gamma)\subseteq
\pi_\Gamma(gH^\Gamma).$$

 It is therefore sufficient to prove:

 \begin{lem}\label{inside-claim} Under the above assumptions, there exists  $\alpha\in X^\sharp$ such that
$A_\alpha=gH$.\end{lem}

\proof[Proof of Lemma.]
 We use induction on $\dim G$, with $\dim G=0$ being a
trivial case. Since left translation by $g^{-1}$ is a polynomial map from $G$ to
$G$, we may replace $X$ by $g^{-1}X$ and assume that the minimal coset containing
$X$ is $H$.

If $H$ is a proper subgroup of $G$ then by induction there exists $\alpha\in X^\sharp$
such that $A_\alpha=H$. Thus, we may assume that $H=G$, and we wish to find
$\alpha\in X^\sharp$ such that the nearest coset to $\alpha$ is $G$. We define $\alpha$
as follows:

We choose $\beta=(\beta_1,\ldots,  \beta_d)\in {\mathfrak R}^d$ with
$0 <<\beta_1<<\beta_2<< \cdots <<\beta_d$. By that we mean $\beta_1> \RR$,
and  for every $i=1,\ldots,
d-1$, and  every polynomial $q(x_1,\ldots, x_i)\in \RR[x_1,\ldots, x_i]$ we have
$\beta_{i+1}>q(\beta_1,\ldots, \beta_i)$.  We can find such a tuple $\beta$ because
$\mathfrak R$ is $|\RR|^+$-saturated. The following is easy to verify:
\begin{claim}\label{beta1} If $q(x_1,\ldots,x_d)\in \RR[x_1,\ldots, x_d]$ is
a non-constant polynomial then $q(\beta)\notin \CO(\mathfrak R)$.\end{claim}

We now claim that $\alpha=f(\beta)$ is the desired element. Towards that we prove
the following general claim:
\begin{claim}\label{polynomial} For $\beta\in {\mathfrak R}^d$ and $G$ as above,  if $q:\RR^d\to G$
is a polynomial map, and $gH_0$ is near $q(\beta)$, for some real algebraic $H_0\subseteq
G$ and $g\in G$, then $q(\beta)\in gH_0$.
\end{claim}

 Before proving the claim let us see that it implies Lemma~\ref{inside-claim}. Indeed, the
above claim implies that when $gH_0$ is any coset near $\alpha$ then $\alpha\in
gH_0$. We now consider the set $S=\{x\in \RR^d: q(x)\in gH_0\}$. Since $H_0$ is a
real algebraic group, the set $S$ is also real algebraic, defined over $\RR$.  The
transcendence degree of $\beta$ over $\RR$ is $d$, and since $\alpha \in H_0$ and
$\beta\in S^\sharp$,  we must have $S=\RR^d$. It follows that $X\subseteq gH_0$, and
therefore the nearest coset to $\alpha$ must contain $X$.  By our assumptions, it
follows that $A_\alpha=G$, thus ending the proof of Lemma~\ref{inside-claim}, and
with it the proof of Theorem~\ref{Leibman-Shah}.

Thus, we are left to prove Claim~\ref{polynomial}, and we do so by induction on the
$\dim G$.  We may assume that $gH_0$ equals $A_\alpha$, and by replacing the map $q$
with the polynomial map $g^{-1}q$, we may assume that the group $A_\alpha=H_0$. We
want to show that $\alpha\in H_0$. Without loss of generality, $H_0$ is a proper
subgroup of $G$, for otherwise we are done.

We may further assume that there is no proper algebraic subgroup $H_1\subseteq G$ such
that $q(\RR^d)\subseteq H_1$ (for otherwise $H_0$ is also contained in $H_1$ and we may
replace $G$ with $H_1$ and finish by induction). Let $N$ be a proper real algebraic
normal subgroup of $G$ containing $H_0$ and consider the map $\pi\circ q$, where
$\pi:G\to G/N$ is the quotient map.  By Lemma~\ref{homom2} (1), the map $\pi\circ q$
 is still
polynomial, and by our assumptions the trivial group $\{e\}$ is near $\pi\circ
q(\beta)$, and in particular $q(\beta)\in \CO(\mathfrak R)$. By Claim~\ref{beta1},
the map $\pi\circ q$ must be a constant map, which is necessarily  $e$. It follows
that  $q(\RR^d)\subseteq N$, contradicting our assumption.
 This
ends the proof
 of Claim~\ref{polynomial} and with it the proofs
of Lemma~\ref{inside-claim} and Theorem~\ref{Leibman-Shah}.
\end{proof}

\section{Neat families of cosets}
The work here is similar to the work in \cite[Section 7.1-7.2]{o-minflows}. We assume that $G$ is a real unipotent group.

Our first goal is to show that the family $\CA(X)$ of all nearest cosets to elements
in $X^\sharp$, is an $\Rom$-definable subfamily of the family of all cosets of real
algebraic subgroups of $G$ (see Fact~\ref{all-cosets}). This is very similar to the
work in \cite{mustab}. We expand the structure $\mathfrak R_{\mathrm{om}}$ by adding
a predicate symbol for the set of reals $\RR$. We are thus working in the structure
$\mathfrak R_{pair}=\la \mathfrak R_{\mathrm{om}},\Rom\ra$, in which $\Rom$ is an
elementary substructure of $\mathfrak R_{\mathrm{om}}$. Such structures are called
tame pairs of o-minimal structures and were studied in \cite{lou-limit}.

Note first that since the standard part map is definable in $\mathfrak R_{pair}$,
the family $\CA(X)$ is definable in $\mathfrak R_{pair}$. By \cite[Proposition
8.1]{lou-limit} we may conclude:

\begin{lem}\label{definability} The family of cosets $\CA(X)$ is definable in $\RR_{\mathrm{om}}$. Namely, there exists in $\Rom$ a definable set $T$ and a formula
$\phi(x,t)$, with $x$ and $t$ tuples of variables, such that
$$\CA(X)=\{\phi(G,t):t\in T\}.$$
\end{lem}

Our next goal is to replace $\CA(X)$ by a family of cosets of finitely many
subgroups.

\begin{defn} Let $\CF=\{g_tH_t:t\in T\}$ be an $\Rom$-definable family of cosets of real algebraic subgroups of $G$. We
say that $\CF$ is \emph{neat} if the following hold:

\begin{enumerate}
\item For $t_1\neq t_2$, $g_{t_1}H_{t_1}\neq g_{t_2}H_{t_2}$.
\item There exists $k$, such that $T$ is a connected submanifold  of $\RR^k$.
 \item There exists a definable continuous function from $T$ to $G$, $t\mapsto h_t\in
 G$,  such that for every $t\in T$, $h_tH_t=g_tH_t$.
\item For every nonempty open $U \subseteq T$,
$$\la \bigcup_{t\in U}H_t\ra_{alg}=\la \bigcup_{t\in T} H_t\ra_{alg}.$$
\end{enumerate}

\medskip
For $\CF$ a neat family of of cosets as above, we denote by $H_\CF$ the group $\la
\bigcup_{t\in T}H_t\ra_{alg}$.
\end{defn}

\begin{lem}\label{neat1} Let $\CF$ be a neat family of algebraic
  subgroups
of $G$. Then for every  lattice $\Gamma\subseteq G$, the set $T_\Gamma=\{t\in
T:H_t^{\Gamma}=(H_\CF)^\Gamma\}$ is dense in $T$.
\end{lem}
\begin{proof}  For a $\Gamma$-rational subgroup $L$ of $G$, let
$$T(L)=\{t\in T: H_t\subseteq
L\}.$$

Clearly, if $t\in T\setminus T_\Gamma$ then $H_t^\Gamma$ is a proper subgroup of
$(H_\CF)^\Gamma$, hence $T\setminus T_\Gamma$ can be written as  a union of all sets
$T(L)$, as $L$ varies over all $\Gamma$-rational proper subgroups of
$(H_\CF)^\Gamma$.

 By Remark~\ref{rem:uniform}, there are countably many $\Gamma$-rational
subgroups of $G$, thus the union is countable. So, in order to show that $T_\Gamma$
is dense in $T$ it is sufficient, by Baire Category Theorem, to show that every
$T(L)$ is nowhere dense. Since this is a definable set it is sufficient to prove
that $T(L)$ does not contain any nonempty open subset of $T$. But, by definition of
$H_\CF$, for every $U\subseteq T$ nonempty open set, the group  $\la \bigcup_{t\in
U}H_t\ra_{alg}$  is the whole of $H_\CF$, so  $\la \bigcup_{t\in
U}H_t\ra_{alg}^\Gamma =(H_\CF)^\Gamma$. On the other hand, for every  $V\subseteq
T(L)$, we have $\la \bigcup_{t\in V}H_t\ra_{alg}^\Gamma\subseteq L\neq (H_\CF)^\Gamma$,
so no open nonempty subset of $T$ is contained in $T(L)$. Therefore, $T_\Gamma$ is
indeed dense in $T$.\end{proof}

\begin{lem}\label{partition} Let $\{g_tH_t:t\in T\}$ be a definable family of pairwise distinct cosets
 of algebraic subgroups of $G\subseteq \UT(n,\RR)$. Then \begin{enumerate} \item there is a definable
partition of $T=T_1\cup\cdots\cup T_r$, such that for each $i=1,\ldots, r$ the
family $\{g_tH_t:t\in T_i\}$ is neat. \item For each $i=1,\ldots, r$, let
$$L_i=\la\bigcup_{t\in T_i} H_t\ra_{alg}.$$
Then
 for every lattice $\Gamma\subseteq G$,
 $$\cl(\bigcup_{t\in T_i} g_tH_t^\Gamma)= \cl(\bigcup_{t\in T_i}g_tL_i^\Gamma).$$
 \end{enumerate}\end{lem}
\begin{proof} (1) We use induction on $\dim T$.
 By o-minimality, we may assume that $T$
is a connected submanifold of some $\RR^k$ and that the function $t\mapsto g_t$ is
continuous on $T$. Given $t\in T$, it follows from DCC for real algebraic subgroups
that there exists a subgroup $G_t\subseteq G$ such that for all sufficiently small open
$t\in U\subseteq T$, $\la \bigcup_{t\in U}H_t\ra_{alg}=G_t$.

Because the family of all real algebraic subgroups of $G$ is definable the family
$\{G_t:t\in T\}$ is  also definable, thus we may divide $T$ into finitely many
definable submanifolds, $T_1,\ldots, T_m$, on each of which $\dim G_t$ is constant.
By induction, it is sufficient to handle those $T_i$ whose dimension equals that of
$T$. Notice that for such a $T_i$, and $t\in T_i$, it is still the case that for all
sufficiently small open $U\subseteq T_i$, a neighborhood of $t$, we have
$$G_t=\la \bigcup_{t\in U}H_t\ra_{alg}$$ (this might not be the case for those
$T_i$'s with $\dim T_i<\dim T$).

 Thus, without loss of generality, $\dim
G_t$ is constant as $t$ varies in $T$.
 We claim that now the group
$G_t$ is the same for all $t\in T$ (and hence $\{g_tH_t:t\in T\}$ is a neat family).
Indeed, fix $t_0\in T$ and let
$$T_0=\{t\in T;G_t=G_{t_0}\}.$$

The set $T_0$ is closed in $T$: Let $t_1\in \cl(T_0)$ and fix $U\ni t_1$ such that
$G_{t_1}=\la \bigcup_{t\in U}H_t\ra_{alg}$. For every $t\in U\cap T_0$, we have
$G_{t}=G_{t_0}\subseteq G_{t_1}$, but since $\dim G_t$ is constant in $T$ we must have
$G_{t_1}=G_{t_0}$, so $t_1\in T_0$.

Let us see that $T_0$ is also open in $T$. For $t_2\in T_0$ let $t_2\in U\subseteq T$ be
an open set such that $G_{t_2}=G_{t_0}=\la \bigcup_{t\in U}H_t\ra_{alg}$. By
dimension considerations, for all $t\in U$, $G_t=G_{t_0}$, so $U\subseteq T_0$, and thus
$T_0$ is open.

 Because $T$ is connected, $T_0=T$. It follows that for every open
nonempty sets $U\subseteq T$
$$\la \bigcup_{t\in U} H_t\ra_{alg}=\la \bigcup_{t\in T} H_t\ra_{alg}.$$

(2) Fix $i=1,\ldots, r$ so  the family $\{g_tH_t:t\in T_i\}$ is neat. First note that for $t\in T_i$, each
$g_tH_t^\Gamma$ is contained in $g_tL_i^\Gamma$, so it is sufficient to show that
$\bigcup_{t\in T}g_tH_t^\Gamma$ is dense in $\bigcup_{t\in T}g_tL_i^\Gamma$.

 By Lemma~\ref{neat1}, the set $T_0=\{t\in T:H_t^\Gamma=L_i^\Gamma\}$ is dense in $T_i$. Let
$g_{t_0}h_0$ be an arbitrary element of $g_{t_0}L_i^\Gamma$, for some $t_0\in T_i$,
and choose $t_n\in T_0$ a sequence converging to $t_0$. For each $t_n$ we have
$g_{t_n}h_0\in g_{t_n}L_i^\Gamma=g_{t_n}H_{t_n}^\Gamma$. Because the map $t\mapsto
g_t$ is continuous, $g_{t_n}h_0$ tends to $g_{t_0}h_0$, so indeed the union of
$g_tH_t^\Gamma$ is dense in the union of $g_tL_i^\Gamma$.\end{proof}

\section{The main theorem}
We are now ready to prove Theorem~\ref{thm-main2.5}. We find it convenient to
reformulate the result within $G$ and not in $G/\Gamma$. The equivalence of the
theorem below to Theorem~\ref{thm-main2.5} follows from the definition of the
quotient topology on $G/\Gamma$. Namely, for  every $X\subseteq G$,  $\pi_\Gamma(X)$ is
closed in $G/\Gamma$ if and only if $X\Gamma$ is closed in $G$.

All definability below is taken in the o-minimal structure $\Rom$.

\begin{thm}\label{thm-main3} Let $G$ be a real unipotent group
and let $X\subseteq G$ be a definable set. Then there are finitely many definable real
algebraic subgroups $L_1,\ldots, L_m\subseteq G$ of positive dimension, and finitely many
definable closed sets $C_1,\ldots, C_m\subseteq G$, such that for every lattice
$\Gamma\subseteq G$,
$$\cl(X\Gamma)=(\cl(X) \cup \bigcup_{i=1}^m C_i L_i^\Gamma) \Gamma.$$ In addition,
the $C_i$'s can be chosen to satisfy:
\begin{enumerate} \item For every $i=1,\ldots, m$,
$\dim(C_i)<\dim X$.

\item Let $L_i$ be a  maximal subgroup with respect to inclusion, among $L_1,\ldots,
L_m$. Then  $C_i$ is a bounded set in $G$, and in particular $C_iL_i^\Gamma \Gamma$
is closed in $G$.
\end{enumerate}
\end{thm}
\begin{proof} Recall that for a coset $A=gH\subseteq G$, and a lattice $\Gamma$, we write
$A^\Gamma$ for $gH^\Gamma$.  In particular, $\cl(A\Gamma)=A^\Gamma \Gamma$.

 By Corollary~\ref{maincor-cosets},
$$\cl(X\Gamma)=\st(X^\sharp\Gamma^\sharp)=
\bigcup_{A\in \mathcal A(X)}A^\Gamma \Gamma.$$

By Lemma~\ref{definability}, the family of cosets $\CA(X)$ is definable in $\Rom$.
By Definable Choice, we may assume that the cosets in $\CA(X)$ are pairwise
distinct. As we already pointed out, the zero-dimensional cosets in this family are
exactly the singletons of elements of $X$. Thus we restrict our attention to those
cosets which have positive dimension and denote this definable sub-family by
$\CA(X)'$.

By Lemma~\ref{partition}, we can divide $\CA(X)'$ into finitely many neat families
of cosets, $\CA_1\cup\cdots\cup \CA_m$. For each $i=1,\ldots, m$, the family
$\CA_i=\{g_tH_t:t\in T_i\}$ has an associated fixed group $L_i=\la \bigcup_{t\in
T_i} H_t\ra_{alg}$.  By Lemma~\ref{partition}, for every lattice
$\Gamma \subseteq G$ and for each $i=1,\ldots, m$, we have
$$\bigcup_{A\in \CA_i} \cl(A\Gamma)=\bigcup_{t\in
T_i}\cl(g_tH_t^\Gamma)\Gamma=\bigcup_{t\in T_i}g_tL_i^\Gamma \Gamma.$$

For each  $i=1,\ldots, m$ we consider the group $L_i$. By Lemma~\ref{complement-to-group}, for each $i=1,\ldots, m$, there exists a closed
semi-algebraic ``complement'' $A_i\subseteq G$,
 to the group $L_i$.
Namely, the map $(a,h)\to ah$ is a diffeomorphism of $A_i\times L_i$ and $G$. We let
$(a_i, h_i):G\to A_i\times L_i$ be its inverse map, so for every $g\in G$ we have
$g=a_i(g)h_i(g)$. Notice that the map $a_i$ is constant on left cosets of $L_i$.

Since the map $a_i\colon G\to A_i$  is continuous, we may replace the map $t\mapsto
g_t$ on $T_i$ by the continuous map $t\mapsto a_i(g_t)$ and thus assume, for each
$I=1,\ldots, m$, that $g_t$ takes value in $A_i$. By our choice of $\CA(X)'$, it is
also injective. We let $C_i=\cl(\{g_t:t\in T_i\})$ (there is no harm in taking
closure since we are describing closed set $\cl(X\Gamma)$). So, $C_i\subseteq A_i$.

 Thus,
$$\cl(X\Gamma)=\cl(X)\Gamma \cup \bigcup_{i=1}^m \bigcup_{A\in \CA_i}\cl(A\Gamma)=
(X\cup  \bigcup_{i=1}^mC_iL_i^\Gamma)\Gamma.$$ This ends the proof of the main
result.

Let us see that our sets $C_i$ satisfy (1) and (2). It is sufficient to prove both
for $C_i'=\{g_t:t\in T_i\}$ instead of $C_i=\cl(C_i')$. Indeed, by o-minimality
$\dim C_i'=\dim C_i$ and clearly $C_i$ is bounded if and only if $C_i'$ is.
\medskip

 (1)We need  to show that $\dim C_i'<\dim X$.
 By our choice of $T_i$ and $C_i'$, for each $g\in C_i'$ there exists $\alpha\in G^\sharp
\, \setminus \, \CO(G)$ such that $A_\alpha\subseteq gL_i$.  In particular, the coset
$gL_i$ is near $\alpha$.

 Recall that $G$ is a closed subset of  $\RR^{n^2}$ and
$\CO(G)$ is the collection of all elements of $G$ which are $\RR$-bounded. Given $g\in
G$ we let $|g|$ be its Euclidean norm as an element of $\RR^m$. As we noted in
Section~\ref{standard-part}, for $\alpha\in G^\sharp$, $\alpha\in \CO(G)$ if and
only if $|\alpha|\in \CO(\mathfrak R)$.

 We define
$$X_i=\{(a_i(x),1/|h_i(x)|)\in A_i\times \RR: x\in X\}.$$

The set $X_i$ is definable and there is clearly a definable surjection from $X$ onto
$X_i$, thus $\dim X\geq \dim X_i$.

\medskip

\noindent{\bf Claim} If $g\in C_i'$ then $(g,0)$ is in $Fr(X_i)=\cl(X_i)\setminus
X_i$.

\begin{proof}[Proof of Claim.]
Clearly, $(g,0)\notin X_i$, so we need to see that it belongs to $\cl(X_i)$.

First note that since the map $(a_i,h):G\to A_i\times L_i$ is a semialgebraic
homeomorphism over $\RR$, it sends $\CO(G) $ onto $(\CO(G)\cap A_i^\sharp)\times
(\CO(G)\cap L_i^\sharp)$. Next,
 as we noted above, there exists $\alpha\in X^\sharp\setminus \CO(G)$ such that the coset $gL_i$ is near
$\alpha$.

 So, there exists $\epsilon\in \mu(G)$ such that $\alpha\in \epsilon
gL_i^\sharp$. Since $\alpha$ and $\epsilon g$ are in the same left coset of
$L_i^\sharp$, we have $a_i(\epsilon g)=a_i(\alpha)$. Because $a_i(-)$ is a
continuous map, and $a_i$ is the identity on $A_i$, we have
$$\st(a_i(\epsilon g))=a_i(g)=g,$$ and in particular, $a_i(\alpha)\in \CO(G)$ and $\st(a_i(\alpha))=g$.

We have $\alpha=a_i(\alpha)h_i(\alpha)$, and since $\alpha\notin \CO(G)$ and
$a_i(\alpha)\in \CO(G)$, then $h_i(\alpha)\notin \CO(G)$, so $|h_i(\alpha)|\notin
\CO(\mathfrak R)$, hence $\st(1/|h_i(\alpha)|)=0$. Thus,
$(g,0)=(\st(a_i(\alpha)),\st(1/ |h_i(\alpha)|))$ is in $\st(X_i^\sharp)$, which by
Fact~\ref{standard}, equals $\cl(X_i)$.\end{proof}

\medskip
By o-minimality,  $\dim Fr(X_i)<\dim X_i\leq \dim X$, so it follows from our Claim
that $\dim C_i'<\dim X$.

\medskip

(2) We may assume that the groups $L_1,\ldots, L_r$ are maximal with respect to
inclusion among $L_1,\ldots,L_m$ (note that we allow repetitions among the $L_i$'s).
We first prove:
\begin{claim}\label{claim:Bexits}
  There is a definable closed bounded set  $B\subseteq G$ such that
  \[ X \subset  BL_1\cup \dotsb \cup BL_r. \]
\end{claim}
\begin{proof}[Proof of Claim.]
  Our construction implies that for every $\alpha \in
X^\sharp\,\setminus\,\CO(G)$, if $A_\alpha=g_\alpha H_\alpha$ then there exists
$i\in\{1,\ldots, m\}$  and  $g\in C_i'$ with  $A_\alpha\subseteq gL_i$,  hence $\alpha\in
\CO(G)L_i^\sharp$. Each $L_i$ is contained in some $L_j$, with $1\leq j\leq r$, and
hence
\[ X^\sharp \subseteq \CO(G) \cup \bigcup_{i=1}^r \CO(G)L_i^\sharp. \]
Writing $\CO(G)$ as a countable union of definable closed bounded sets
and using  the
Compactness Theorem (in Logic) we obtain that there is a definable
closed bounded set
$B\subseteq G$ with
\[ X \subseteq B \cup \bigcup_{i=1}^r B L_i. \]

If $X$ is bounded then $r=m=0$ and then $X\subseteq B$ for some $B$. Otherwise, $B\subseteq
BL_i$ for every $i$, and hence
\[ X \subseteq \bigcup_{i=1}^r B L_i. \]
 This proves Claim~\ref{claim:Bexits}.
\end{proof}

We fix a set $B$ as in Claim~\ref{claim:Bexits}.
\begin{claim}\label{claim:Bexits1}
  For every $\alpha\in X^\sharp$ there is $b\in B$ and $i\in\{1,\dots,
  r\}$ such that $A_\alpha \subseteq b L_i$, and in particular, $H_\alpha\subset
L_i$
\end{claim}
\begin{proof}[Proof of Claim]
Let $\alpha\in X^\sharp$. It follows from Claim~\ref{claim:Bexits} that there is
$b\in B$ and $i\in \{1,\dotsc, r\}$ such that $\alpha$ is
  near the coset $bL_i$. (If $\alpha\in B^\sharp$, then $\alpha$ is near
  the coset $bL_1$, where $b=\st(\alpha)\in B$).
This proves Claim~\ref{claim:Bexits1}.
\end{proof}

We now proceed with the proof of $(2)$ and fix a maximal $L_i$.  Without loss of
generality, $i=1$.

We need to show that $C_1'$ is bounded. So assume towards getting a contradiction
that $C_1'$ is unbounded.

It is not hard to see that there is a bounded closed definable set $B_1 \subseteq
A_1$ (recall $A_1$ is the complement of $L_1$) such that $B\subseteq B_1L_1$,
hence $BL_1 \subseteq B_1L_1$. Because $C_1'$ is unbounded subset of $A_1$, we have
$C_1'\not \subseteq B_1$.

Thus, by our choice of $C_1'$ and $L_1$, there is a neat family $\CF=\{ g_t H_t
\colon t\in T_1\}$ (with $g_t$ taking values in $A_1$), such that: (i)
$H_{\CF}=L_1$, (ii) for every $t\in T_1$ there is $\alpha\in X^\sharp$ with
$A_\alpha=g_tH_t$ and (iii) for some $t_0\in T_1$, $g_{t_0} \notin B_1$.

By the continuity of $g_t$, there exists an open $U\subseteq T_1$ containing $t_0$ such
that for all $t\in U$, $g_t\notin B_1$. It follows that for all $t\in U$,
$g_tL_1\nsubseteq B_1L_1$ (here we use the fact that $A_1$ contains a single
representative for each left coset of $L_1$), and since $BL_1\subseteq B_1L_1$, we also
have $g_tL_1\nsubseteq BL_1$.

 By Claim~\ref{claim:Bexits1}, the
set $U$ is covered by definable sets
 $S_i$, $i=1,\dotsc,m$, where $S_i=\{ t\in U \colon g_tH_t \subseteq
 BL_i\}$. However, by what we just showed, $U\cap S_1=\emptyset$, so we have
$$U\subseteq \bigcup_{L_i\neq L_1} S_i.$$

It follows from o-minimality that there exists $i_0$, with $L_{i_0}\neq L_1$, such
that $S_{i_0}$ contains
 nonempty open set $U_{i_0}\subseteq U$. Thus, $U_{i_0}\subseteq T_1\cap S_{i_0}$, so for
every $t\in U_{i_0}$,   $H_t$ is
 contained in  $L_1\cap L_{i_0}$. By the maximality of $L_1$, and since $L_1\neq
L_{i_0}$, the group $L_1\cap L_{i_0}$ is a proper subgroup of $L_1$.
  Hence $$\la \bigcup_{t\in U_{i_0}} H_t \ra_{alg}$$
 is a proper subgroup of $L_1$, contradicting  the neatness of the family
 $\CF$. Thus $C_1'$ and therefore $C_1$ is bounded.

 This ends the proof of the clause (2) and
 Theorem~\ref{thm-main3}. \end{proof}

\section{On uniform distribution}
\label{sec:unif-distyr}

In this section we consider the connection between the topological results obtained thus far,   and uniform distributions on nilmanifolds in the case of definable curves.

We work in the structure $\Rom$. By definable we mean
$\Rom$-definable.

We fix a unipotent group $G$, a lattice $\Gamma$ in $G$,
and let  $X=G/\Gamma$ and $\pi\colon G\to X$ be the quotient map.

We denote by $\mu_G(x)$  the unique
Borel regular $G$-invariant probability measure on $X$.

We let   $C^0(X)$ denote the space of all continuous real
valued functions on  $X$, and let $\CP(X)$
be the space of Borel
regular probability measures $X$. We equip $\CP(X)$ we the weak$^*$
topology. With this topology the space $\CP(X)$ is compact,
sequentially compact (i.e. every sequence has a convergent subsequence), and
a sequence $\mu_n$  is convergent to $\mu$ if for any $f\in C^0(X)$ we
have
\[\lim_{n\to \infty} \int_X f(x) d\mu_n(x) =\int_X f(x) d\mu(x). \]

\medskip

By {\em a definable curve on $G$} we mean a continuous definable map
$\gamma\colon \RR^{\geq 0} \to G$. If $\gamma\colon \RR^{\geq 0} \to
G$ is a definable curve then by $[\gamma]$ we denote the image of
$\gamma$, i.e. the set $\gamma(\RR^{\geq 0})$.

For a definable curve $\gamma$  on $G$,  $R\in \RR^{>0}$ and
$f(x)\in C^0(X)$ let
\[ L^\gamma_R(f) = \frac{1}{R}\int_0^R  f{\co} \pi(\gamma(t)) dt.  \]
It is easy to see that $L^\gamma_R$ is a linear functional on
$C_0(X)$ with $L^\gamma_R(\mathbf 1_X)=1$, hence by  Riesz–Markov–Kakutani representation theorem,
there is $\mu^\gamma_R \in \CP(X)$ with
$L^\gamma_R(X)(f)=\int_X f(x) d\mu^\gamma_R(x)$ for all $f\in C^0(X)$.

\medskip

\begin{defn} Let $\gamma$ be a definable curve on $G$.
  \begin{enumerate}
  \item  We say that $\gamma$ is \emph{unbounded} if $\lim_{t\to
      \infty}\gamma(t)$ does not exists in $G$,  equivalently, since $G$ is a
    closed subset of $M_n(\mathbb R)$,  $\lim_{t\to \infty} |\gamma(t)|=+\infty$,
    where $|\gamma(t)|=\|
      \gamma(t)\|_\mathrm{max}$ is the maximal entry, in the absolute value,
    of the matrix $\gamma(t)$.
\item
  We say  that $\gamma$ is \emph{dense} in $G$ mod
$\Gamma$  if the set $[\gamma]\cdot \Gamma$ is dense in $G$,
equivalently $\pi([\gamma])$ is dense in $X$.

\item
We say that $\gamma(t)$ is \emph{continuously
  uniformly distributed (c.u.d. for short)} in $G$ mod $\Gamma$  if
\[ \lim_{R\to \infty} \mu^\gamma_R = \mu_G. \]
  \end{enumerate}
\end{defn}

Finally, recall that an o-minimal structure $\CR$ on $\mathbb R$ is
called {\em polynomially bounded} if every definable function of one
variable which is defined on some positive ray is eventually bounded
by a polynomial over $\mathbb R$. The field structure as well as the
structure $\Ran$ are polynomially bounded (see \cite{polbound})
while $\mathbb R_\mathrm{exp}$ is obviously not polynomially bounded.

Our goal is to prove the following theorem.

\begin{thm}\label{thm:main} Assume the structure $\Rom$ is polynomially
  bounded.  Let $\gamma(t)$ be a definable curve on $G$. If $\gamma$
  is dense in $G$ mod $\Gamma$ then  it is c.u.d. in $G$ mod $\Gamma$.
\end{thm}
\begin{rem} Notice that the converse always holds: if $\gamma$ is
  c.u.d. mod $\Gamma$ then it is dense mod $\Gamma$.

\end{rem}

From now on we assume that $\Rom$ is polynomially bounded.

We prove Theorem~\ref{thm:main} by induction on $\dim(G)$. The
statement is trivial if $\dim(G)=0$. So we assume from now on that $\dim(G)>0$ and
theorem is true for all unipotent groups of dimension smaller than
$G$.

We fix a definable curve   $\gamma(t)$  and assume that it is dense in
$G$ mod $\Gamma$.

Since the space $\CP(X)$ is sequentially compact, Theorem~\ref{thm:main} will follow from
the following theorem.

\begin{thm}\label{thm:main1}  Let $R_k\in \RR^{>0}$ be an unbounded
  increasing sequence such that the limit $\lim_{k\to \infty} \mu^\gamma_{R_k}$, call it $\mu$,
  exists. Then
  $\mu=\mu_G$.
  \end{thm}

  We fix such a sequence $R_k$ and  $\mu=\lim_{k\to \infty} \mu^\gamma_{R_k}$
  We need to show that $\mu$ is $G$-invariant.

  \medskip

Let $p(x)\in S_G(\RR)$ be the type of $\gamma(t)$ ``at infinity'',
namely the complete type determined by  $x\in [\gamma]$ and all formulas $\gamma^{-1}(x)>n$ for  $n\in \mathbb N$.
Since  $\gamma$ is dense in $G$  mod $\Gamma$, the curve $\gamma$ is
unbounded in $G$ (otherwise, by o-minimality,  $\cl([\gamma])$ would be a compact set
and $\cl([\gamma]\cdot \Gamma)=(\cl[\gamma])\cdot \Gamma$ would be a
closed subset of $G$ with empty interior).
By Fact \ref{mustab-thm}, $\Stab^{\mu}(p)$ is a definable subgroup of
$G$ of positive dimension. In fact it has dimension one, but all we
need is that it contains a definable subgroup
$P$ of dimension one.

We first show that the measure $\mu$ is $P$-invariant.

\subsection{The work of Poulios and Shah}
 Our strategy is to apply methods developed by Georgious Poulios in
 his PhD thesis, \cite{Poulios}, and then follow the strategy of
 Nimish~A.~Shah (see \cite[Proposition 4.1]{Shah} for an analogous result for polynomial curves in a more general setting).

Since $\gamma$ is unbounded, we have $\lim_{t\to \infty} |\gamma(t)|=\infty$.
 Since we assume that $\Rom$ is polynomially bounded, by \cite{Miller},
 there is some nonzero $d\in \mathbb R$, and $r>0$ such that $|\gamma(t)|\sim dt^r$, namely $|\gamma(t)|/dt^r\to 1$.

 We now return to the group $P$. It is a $1$-dimensional subgroup of a unipotent group
 $G$, therefore it can be written as an isomorphic image under a
 polynomial map of   a $1$-dimensional additive subgroup of the Lie algebra
 $\mathfrak g$. Thus, $P$ is definably isomorphic to $(\mathbb R,+)$.

   It follows from o-minimality (see \cite{PL}) that for every
   $g\in P$ there is a definable unbounded $h_g:(0,\infty)\to \mathbb
   R$ such that $\lim_{t\to \infty} \gamma(h_g(t))\gamma(t)^{-1}=g$.
   Since $\Rom$ is polynomially bounded, each such $h_g$ has a
   corresponding $e=e(g)\in \mathbb R^{>0}$ and nonzero $c=c(g)\in
   \RR$ such that $h_g(t)\sim ct^e$.

 \begin{claim} For every $g\in P$,

 \begin{enumerate}
 \item $e(g)=1$, namely $h_g(t)\sim ct$.
 \item $c(g)=1$.
 \end{enumerate}

 \end{claim}
 \proof (1) (see also \cite[Lemma 3.2.1]{Poulios}) We have
 $|\gamma(h_g(t))|\sim cdt^{re}$, so if $e\neq 1$ then
 $\gamma(h_g(t))\gamma(t)^{-1}$ is unbounded. Indeed assume $e >1$.
If  $\gamma(h_g(t))\gamma(t)^{-1}$ is bounded then for $M(t)=
\gamma(h_g(t))\gamma(t)^{-1}$ we would have that $|M(t)|\in O(1)$
and then, since $\gamma(h_g(t))=M(t)\gamma(t)$, we
would have $|\gamma(h_g(t))| =O(|\gamma(t)|)$, in contradiction to
 $|\gamma(h_g(t))|\sim cdt^{re}$.  If $e<1$ then using boundedness of
$\gamma(t)\gamma(h_g(t))^{-1}$ we would derive the same contradiction.
Thus $h_g(t)=ct+o(t)$,

 For (2), assume that the limit of $\gamma(ct+o(t))\gamma(t)^{-1}$
 exists (so belongs to $P$) for some $c\neq 1$.
By \cite[Lemma 3.2.2]{Poulios}, for every $a\in \mathbb R^{>0}$, the
limit $\gamma(at)\gamma(t)^{-1}$ exists in $P$, call it $P(a)$, and
moreover, the function $a\mapsto P(a)$ is an isomorphism of $(\mathbb
R^{>0},\cdot)$ and $P$. As noted above, this yields a definable isomorphism of
 $(\mathbb R,+)$ and $\la \mathbb R^{>0},\cdot\ra$, contradicting the fact that $\CR$ is polynomially bounded.\qed

We can therefore write every  $h_g$ as $t+a(g)t^{r_0}+o(t^{r_0})$
for some $r_0<1$ (a-priori, possibly depending on $g$). We now apply \cite[Lemma 3.2.3]{Poulios}, and conclude that
 every element of $P$ can be attained as the limit $\rho(c)=\lim_{t\to \infty}\gamma(T_{r_0,c}(t))\gamma(t)^{-1}$,  where
 $$T_{r_0,c}=(t^{1-r_0}+(1-r_0)c)^{\frac{1}{1-r_0}}=t+ct^{r_0}+o(t^{r_0}),$$
 for some $c\in \mathbb R$.
 Moreover, the function $c\mapsto \rho(c)$ is a group isomorphism of $(\mathbb R,+)$ and $P$.

We therefore reached the following situation: We have a definable group isomorphism $\rho:\mathbb R\to P$, a fixed $r_0<1$,
and for each $s\in \mathbb R$,
$$\rho(s)=\lim_{t\to \infty}\gamma(t+st^{r_0}+o(t^{r_0}))\gamma(t)^{-1}.$$

We now proceed exactly as in \cite[Section 4]{Shah}, with the following elementary lemma replacing \cite[Lemma 4.3]{Shah}:
\begin{lem} Let $\ell:(0,\infty) \to \mathbb R$ be a differentiable
  function such that $\lim_{t \to \infty}\ell'(t)=0$. Then, for any bounded continuous function $f:\mathbb R\to \mathbb R$,
$$\lim_{T\to \infty}\frac{1}{T}\int_1^T [f(t+\ell(t))-f(t)]dt=0.$$
\end{lem}

We conclude, exactly as in \cite[Proposition 4.1]{Shah},
\begin{prop}\label{prop:P-inv} The measure $\mu$ above is invariant under $P$.\end{prop}

\subsection{The proof of Theorem \ref{thm:main1}}

We proceed similarly to the proof of Proposition \ref{main-step}. Let $N$ be the algebraic normal closure
of $P$, namely the smallest normal algebraic subgroup  of $G$
containing $P$.  Let  $N^\Gamma$ be the $\Gamma$-rational closure of $N$, namely  the  minimal
$\Gamma$-rational subgroup of $G$ containing $N$. Thus $N^\Gamma$
is the smallest $\Gamma$-rational normal subgroup containing $P$.
Finally let $N_0=N^{\Gamma} \cap Z(G)$.  By the nilpotency of $G$, it is an algebraic
$\Gamma$-rational subgroup
of $G$  of positive dimension.

Let $G_1=G/N_0$, and $f\colon G\to G_1$ the quotient map.
Let $\Gamma_1=f(\Gamma)$,  $X_1=G_1/\Gamma_1$, and
$\pi_1\colon G_1\to X_1$  the natural  projection.  Let
$f'\colon X\to X_1$ be the map such that the following diagram is
commutative.

\[
  \begin{tikzcd}
    G \arrow[r, "f"] \arrow[d,"\pi"] & G_1 \arrow[d,"\pi_1"]\\
   X \arrow[r, "f' "]  & X_1
\end{tikzcd}
\]

Since $N_0$ is a central subgroup of $G$, $X_1$ can also be identified
with the quotient $G/N_0\Gamma$, and
the map $f'$ identifies  $X_1$ with the
quotient $N_0\backslash X$.

Let $\gamma_1(t)$ be the curve  $(f\co \gamma)(t)$. It is not hard to see that it is dense in
$G_1$ mod $\Gamma_1$.  Hence, by the induction hypothesis, it is
c.u.d. in $G_1$ mod $\Gamma_1$,
namely, the  measure $\lim_{R\to\infty}\mu_R^{\gamma_1}$
is  the unique $G$-invariant probability  measure on $X_1$.

It is not hard to see that each measure $\mu_R^{\gamma_1}$ is the
pushforward of $\mu_R^{\gamma}$ along $f'$ and
$\mu_1$ is the pushforward of  $\mu$ along  $f'$.

Thus we have established the following claim.
\begin{claim}\label{claim:claim-push}
  Let $X_1= N_0\backslash X$ and $f'\colon X\to X_1$ be the quotient map.
Then the pushforward  of $\mu$ along $f'$ is the unique $G$-invariant
probability Borel measure on $X_1$.
\end{claim}

Our next goal is to show that $\mu$ is $N_0$ invariant.
\newcommand{\CH}{\mathcal H}

As in \cite{Shah}, let $\CH$ be the set of all $\Gamma$-rational
subgroups of $G$, and for $H\in \CH$, let
\[ N(H,P) = \{ g\in G \colon g^{-1}Pg\subseteq H \}. \]

\begin{prop}\label{prop:NHP} For every $H\in
\CH$,
  \begin{enumerate}
  \item   $N(H,P)$ is closed and definable, in fact it is semi-algebraic.
  \item   $N(H,P)$ is $N_0$-invariant.
   \item  $N(H,P)$ is invariant under the action of $P$ on the
     left.
  \item   $N(H,P)$ contains an open subset of $G$ if and only if $H$
    contains $N^\Gamma$ if and only if $N(H,P)=G$.
  \end{enumerate}
\end{prop}
\begin{proof}
  $(1)$ is obvious. $(2)$ is trivial since $N_0$ is central.  $(3)$ is
  obvious and $(4)$
  follows from Lemma \ref{intersection}.
  \end{proof}

  Let
  \[ S(P)= \bigcup_{H\in \CH, H\subsetneqq
      N^\Gamma}  N(H,P), \]
  and
  \[ Y =N(N^\Gamma, P) - S(P) = G - S(P).
  \]

In other words $Y$ consists of all $g\in G$ such that  $N^\Gamma$ is
the smallest $\Gamma$-rational subgroup  containing  $g^{-1}Pg$.

Let
\[ T=\pi(Y). \]

Since the family of $\Gamma$-rational subgroups is closed under conjugation by
elements of $\Gamma$, and $N^\Gamma$ is normal,   the set    $S(P)$
is closed under multiplication by elements of $\Gamma$ on the right.
Hence

\[ T= X -\pi(  S(P) ). \]

Also, by Proposition~\ref{prop:NHP}, the set  $S(P)$  is invariant under
multiplication by $N_0$.

Recall that we have a commutative diagram

\[
  \begin{tikzcd}
    G \arrow[r, "f"] \arrow[d,"\pi"] & G_1=G/N_0 \arrow[d,"\pi_1"]\\
   X \arrow[r, "f' "]  & X_1
\end{tikzcd}
\]

Let $S_1=f(S(P))$.
\begin{prop}\label{prop:mus1}
$\mu_1(\pi_1(S_1))=0$.
\end{prop}
\begin{proof} We need to consider 2 cases.\\
  \noindent\textbf{Case 1: the group  $G_1$  is trivial.}
  In this case  $G=N_0$ is abelian, and it is not hard to see that the
  set $S$
  is empty, hence  the set $S_1$ is empty as well.

  \noindent\textbf{Case 2: the group $G_1$  is nontrivial.}
By Proposition~\ref{prop:NHP}, the set  $S$ is a countable
union of closed semi-algebraic nowhere dense $N_0$-invariant subsets
of $G$.

It follows then that $S_1$ is also a countable union of closed nowhere-dense
subsets of $G_1$. Since $S$ is $\Gamma$-invariant on the right, the set $S_1$ is
$\Gamma_1$-invariant, hence  $\mu_1(\pi_1(S_1))=0$.
\end{proof}

Since $\mu_1$ is the pushforward of $\mu$ along $f'$, it follows then
that $\mu(T)=1$ and $\mu(X- T)=0$.

By Proposition~\ref{prop:NHP}(3),  the set $T$ is closet under action
of $P$ on the left.

We now decompose $\mu$ into $P$-ergodic measures:
By the ergodic decomposition theorem, there is a probability
measure $\omega$ on the space $\CP(X)$
supported on  $P$-ergodic measures on $X$ so that for any Borel
$A\subseteq X$ we have
\[ \mu(A)= \int_{\CP(X)} \nu(A) \,d\omega(\nu).\]
Since $\mu$ is $0$ outside of $T$ and $T$ is $P$-invariant,
we need to consider only $P$-ergodic
measures $\nu$ with $\nu(T)=1$.

By Ratner's theorem (see \cite{Ratner}) for any such $\nu$ there is $g\in T$
such that $\nu$ is $gHg^{-1}$ invariant measure on the closed orbit
$gH\Gamma/\Gamma$, where $H$ is the smallest
$\Gamma$-rational subgroup containing  $g^{-1}Pg$.  Since $g\in T$, the smallest
$\Gamma$-rational subgroup containing  $g^{-1}Pg$ is $N^\Gamma$ that
is normal in $G$ and contains  $N_0$. Thus every such $\nu$ is
$N_0$-invariant, hence $\mu$ in $N_0$-invariant as well.

\medskip

We now have the following situation: the group $G$ acts transitively on a compact
space $X$ and   $N_0$ is a central subgroup of $G$ such that every orbit
$N_0.x$ is closed  in $X$.
The measure  $\mu$ is an $N_0$-invariant
probability Borel measure on $X$ whose push-forward along $f'$ is a
$G$-invariant measure on $X_1=N_0\setminus X$.

Since every nilpotent Lie group is unimodular (left and right Haar
measures are the same), it follows from  \cite[Proposition 1.3]{Furst}
that the measure $\mu$ is $G$-invariant. Thus we conclude that $\mu=\mu_G$.
This ends the proof of Theorem \ref{thm:main}
\qed

\begin{sample}
  The above theorem fails when $\Rom$ is not polynomially bounded, for the group $G=(\mathbb R,+)$.

  Let $\gamma(t)\colon \RR^{\geq 0}\to \RR$ be the curve $\ln(t+1)$,
  and $\Gamma$  the lattice $\Gamma=2\pi \ZZ$ in  $\RR$.
Obviously $\gamma(t)$ is definable in the o-minimal structure
$\RR_{\mathrm{exp}}$ and it is
dense in $\RR$ mod $\Gamma$.  However  it  is not c.u.d. mod
$\Gamma$. To see it,  let
$X=\RR/\Gamma$ and $\pi\colon \RR\to X$ be the quotient map. The function
$\sin(x)\colon \RR\to \RR$ is $\Gamma$-invariant, hence it   induces
a function on $X$, i.e. there is a continuous function  $f\colon X\to
\RR$ such that $ (f\circ \pi)(x)=\sin(x)$. Since
\[  \int \sin(\ln x)=\frac{x}{2}(\sin(\ln x) - \cos(\ln x))+C, \]
we have
  \begin{multline*}
  L_R^\gamma(f)=\frac{1}{R} \int_0^R f\circ \pi(\gamma(t)) dt
  =\frac{1}{R} \int_0^R \sin(\ln(t+1)) dt\\
  =\frac{R+1}{2R}\bigl(\sin(\ln(R+1))-\cos(\ln(R+1))\bigr) +\frac{1}{2R}
  \end{multline*}
that does not have limit as $R$ goes to $\infty$.
\end{sample}
\begin{rem}The failure of the equidistribution in the above example is due to the failure of  Proposition \ref{prop:P-inv}. Indeed, even when the limit  $\mu=\lim_{k\to \infty} \mu^\gamma_{R_k}$ exists for some unbounded sequence $R_k$, then  $\mu$ need not be $P$-invariant, or in this case $(\mathbb R,+)$-invariant.
\end{rem}

\bibliographystyle{acm}

\end{document}